\documentclass[reqno]{amsart}
\usepackage{amsmath,amsfonts,amssymb,graphics,graphicx,latexsym,amscd, subfigure}
\usepackage{IEEEtrantools}
\usepackage{hyperref}
\usepackage[usenames]{xcolor}

\newtheorem{thm}{Theorem}[section]
\newtheorem{theorem}[thm]{Theorem}
\newtheorem{cor}[thm]{Corollary}
\newtheorem{prop}[thm]{Proposition}

\newtheorem*{question*}{Question}
\newtheorem{lemma}[thm]{Lemma}
\newtheorem{defn}[thm]{Definition}
\newtheorem{rem}[thm]{Remark}
\newtheorem{remark}[thm]{Remark}

\newcommand{\bbR}{\mathbb{R}}

\newcommand{\bbT}{\mathbb{T}}
\newcommand{\bbC}{\mathbb{C}}

\newcommand{\bbZ}{\mathbb{Z}}
\newcommand{\bbH}{\mathbb{H}}

\newcommand{\cS}{\mathbb{S}}

\newcommand{\gein}{g^E}
\newcommand{\M}{M}

\newcommand{\mA}{{\mathcal A}}
\newcommand{\mB}{{\mathcal B}}
\newcommand{\mC}{{\mathcal C}}
\newcommand{\mD}{{\mathcal D}}
\newcommand{\mO}{{\mathcal O}}

\newcommand{\R}{\mathbb R}

\newcommand{\N}{\mathbb N}

\newcommand{\CP}{\mathbb {CP}}
\newcommand{\T}{\mathbb T}

\newcommand{\del}{\partial}

\newcommand{\curv}{{c}}

\DeclareMathOperator{\Lap}{\triangle}

\DeclareMathOperator{\Hess}{Hess}

\DeclareMathOperator{\scal}{Scal}

\newcommand{\cp}{{\mathbb{CP}}}

\newcommand{\pol}{{\bf{p}}}

\begin{document}
	
\title{Einstein metrics via Derdzi\'nski duality}

\author[Gon\c calo Oliveira]{Gon\c calo Oliveira}
\address{Gon\c calo Oliveira, Department of Mathematics and Center for Mathematical Analysis, Geometry and Dynamical Systems, Instituto Superior T\'ecnico, Lisbon\\ Av. Rovisco Pais, 1049-001 Lisboa,  Portugal}
\email{goncalo.m.f.oliveira@tecnico.ulisboa.pt}
\email{galato97@gmail.com}

\author[Rosa Sena-Dias]{Rosa Sena-Dias}
\address{Rosa Sena-Dias, Department of Mathematics and Center for Mathematical Analysis, Geometry and Dynamical Systems, Instituto Superior T\'ecnico, Lisbon\\ Av. Rovisco Pais, 1049-001 Lisboa,  Portugal}
\email{senadias@math.ist.utl.pt}

\begin{thanks}
{Partially supported by the Funda\c{c}\~{a}o para a Ci\^{e}ncia e a Tecnologia (FCT/Portugal) through project EXPL/MAT-PUR/1408/2021EKsta.}
\end{thanks}

\begin{abstract}
Drawing on results of Derdzi\'nski's from the 80's, we classify conformally K\"ahler, $U(2)$-invariant, Einstein metrics on the total space of $\mO(-m)$, for all $m \in \mathbb{N}$. This yields infinitely many $1$-parameter families of metrics exhibiting several different behaviours including asymptotically hyperbolic metrics (more specifically of Poincar\'e type), ALF metrics, and metrics which compactify to a Hirzebruch surface $\mathbb{H}_m$ with a cone singularity along the ``divisor at infinity''. This allows us to investigate transitions between behaviours yielding interesting results. 
For instance, we show that a Ricci--flat ALF metric known as the Taub-bolt metric can be obtained as the limit of a family of cone angle Einstein metrics on $\cp^2 \# \overline{\cp}^2$ when the cone angle converges to zero.

We also construct Einstein metrics which are asymptotically hyperbolic and conformal to a scalar-flat K\"ahler metric. Such metrics cannot be obtained by applying Derdzi\'nski's theorem.
\end{abstract}

\maketitle

\tableofcontents

\section{Introduction}
In \cite{d}, Derdzi\'nski gave a method to construct Einstein metrics which are conformally K\"ahler. He proved that a $4$-dimensional Bach-flat K\"ahler metric whose scalar curvature is not identically zero is, at least locally, conformal to an Einstein metric. In light of this result, it is natural to search for Bach-flat K\"ahler metrics and study the Einstein metrics they give rise to through Derdzi\'nski's theorem.

Extremal K\"ahler metrics were introduced by Calabi as a canonical choice for a metric in a K\"ahler class. They are critical points for the $L^2$ norm of the Riemann tensor restricted to that K\"ahler class i.e. the Calabi functional. Although recent results make it so that we have a satisfying general theory for existence and uniqueness of extremal metrics (see \cite{cc} for instance), especially in the constant scalar curvature (cscK) case, there are few explicit examples of extremal K\"ahler metrics. Most known examples arise from the Calabi's ansatz (see \cite{c}), which in dimension $4$ yields metrics with a $U(2)$-symmetry. 

Concerning Bach-flat K\"ahler metrics, which are automatically extremal, even less is known. Chen, Lebrun and Weber discuss K\"ahler metrics with vanishing Bach tensor with an eye on applying Derdzi\'nski's construction on (compact) Del Pezzo surfaces (see \cite{clw}). In their article, the authors obtain a very interesting characterisation of Bach-flat metrics: these extremize the Calabi functional in a K\"ahler class which is itself extremal for the Futaki invariant. 
In general, it is hard to find which K\"ahler classes support Bach-flat metrics. 

The goal of the present article is to use Calabi's ansatz together with Derdzi\'nski's results to construct explicit examples of non-K\"ahler Einstein metrics and study their moduli. Our starting point is an existence result for the Bach-flat K\"ahler metrics which we will use to determine the conformal class on which we find Einstein metrics. Below we present an informal version of our classification of Bach--flat K\"ahler metrics with $U(2)$-symmetry on the total space of $\mathcal{O}(-m)\to\CP^1$. For the precise version of the statement see Theorem \ref{prop:Globalizing} later in the text.

\begin{thm}[Classification of $U(2)$-invariant Bach-flat K\"ahler metrics on $\mO(-m)$]
	
There is an explicit $1$-parameter family of Bach-flat K\"ahler metrics on the total space of $\mO(-m)\rightarrow \cp^1.$ When such a metric is complete, then it either has: 
\begin{itemize}
	\item exponential volume growth, or
	\item quartic volume growth, or
	\item finite volume.
\end{itemize}
On the other hand, if the metric is incomplete, it either:
\begin{itemize}
	\item extends as a metric on the Hirzebruch surface $\mathbb{H}_m=\mathbb{P}(\mathcal{O} \oplus \mathcal{O}(-m))$ with a cone angle along the $\cp^1$ at infinity, or
	\item develops another singular behaviour.
\end{itemize} 
Furthermore, all such behaviours occur and any $U(2)$-invariant Bach-flat K\"ahler on $\mO(-m)$ is one of these.
\end{thm}

Our method builds on the Calabi ansatz. It unifies and contains the approaches taken by several authors on recent papers (\cite{Naff,Gau,Gut}). On the other hand, our result intersects the work in \cite{Bryant,Apostolov} concerning Bach-flat K\"ahler metrics which are in addition self-dual, but without symmetry assumptions. 

After proving the classification result above, we study the vanishing locus of the scalar curvature of such metrics since Derdzi\'nski's construction only gives an Einstein metric on the complement of this vanishing locus. The goal is to understand the underlying space where the Einstein metric is defined. 
The case $m=1$ is rather exceptional as the scalar curvature of the Bach-flat K\"ahler metric is nowhere vanishing. We show the following.
\begin{thm}
For every cone angle $\beta$ smaller than $4\pi,$ there is an explicit Einstein metric on $\cp^2 \# \overline{\cp^2}$ with that cone angle along a 2-sphere $\Sigma\subset \cp^2 \# \overline{\cp^2}.$

In addition, as $\beta \to 0$, such metrics converge uniformly with all derivatives to the Taub-bolt metric which is a complete Ricci-flat ALF metric on $(\mathbb{CP}^2 \# \overline{ \mathbb{CP}^2} ) \backslash \Sigma$.
\end{thm}

One of the consequences of our work is that we are able to explicitly connect metrics with very distinct behaviours into families. Similar transitions had been previously predicted and observed in the K\"ahler-Einstein setting (see works of Guenancia \cite{g}, Cheltsov-Rubinstein \cite{cr}, Biquard-Guenancia \cite{bgue}, Rubinstein-Zhang \cite{rz}, Rubinstein-Ji-Zhang \cite{rjz} ). 


For $m\geq 3,$ the Einstein metrics we construct live in non-compact spaces. They are complete and in fact asymptotically hyperbolic or more precisely Poincar\'e-Einstein. Poincar\'e-Einstein metrics have been studied by Fefferman-Graham (see \cite{fg}), Graham-Lee (see \cite{gl}), Anderson (see \cite{an}),  Biquard (see \cite{b}) and more recently  by Gursky-Sz\'ekelyhidi (\cite{gs}) for instance. They appear in the context of the following question.

\begin{question*}
Given a manifold whose boundary is endowed with a fixed conformal structure, is there an Einstein metric on that manifold whose conformal class at the boundary is the given conformal structure?
\end{question*}

In other words: does there exist an Einstein filling of a given conformal structure? There are results establishing local existence and regularity of fillings under certain hypothesis. Our theorem actually gives explicit fillings for certain conformal classes on a Lens space.  Poincar\'e-Einstein metrics such as the one we construct have played an important part in the physics literature. They appear on one side of the AdS/CFT correspondence associating gauge theories on an (anti-De Sitter) manifold $M$ to conformal field theories on the boundary $\partial M$ (see \cite{w}). Page and Pope have written down an explicit  Poincar\'e-Einstein metric on the disk bundle whose conformal class at infinity is that of a Berger metric on $S^3/\bbZ_m.$ We will see that the Page-Pope metric is included in our construction. In fact we show how to ``cap-off" the Page-Pope metric by a Poincar\'e-Einstein metric with a cone angle. This shows that uniqueness of fillings cannot hold among cone angle metrics.
We prove the following result.

\begin{thm}
For every $m\geq 3,$ there are: 
\begin{itemize}
	\item [(a)] two open sets $M_1,M_2 \subset \mathbb{H}_m$ with $\partial \overline{M_1} =  \partial \overline{M_2}$ and $\mathbb{H}_m = M_1 \cup M_2$; 
	\item [(b)] two pairs $\lbrace (g_{M_1})^{\pm}_t \rbrace_{t \in I \subset \mathbb{R}}$ and $\lbrace (g_{M_2})^{\pm}_t \rbrace_{t \in I \subset \mathbb{R}}$ of $1$-parameter families of Poincar\'e-Einstein metrics on $M_1$ and $M_2$ respectively, sharing the same conformal infinity 
	$$[(g_{M_1})^{\pm}_t |_{\partial M_1}] = [(g_{M_2})^{\pm}_t |_{\partial M_2} ],$$ 
	for each $\pm \in \lbrace + , - \rbrace$ and $t \in I \subset \mathbb{R}$.
\end{itemize} 
One of the families consists of smooth metrics while the other consists of metrics with a cone angle along a surface with self-intersection $m$ (the cone angle depends on $\pm$ and $t$). 
\end{thm}

The boundary of $M_1$ and $M_2$ is diffeomorphic to $S^3/\bbZ_m,$ and the family of smooth metrics coincides with those found by Page--Pope \cite{pp}. 

Finally, we also obtain a family of $U(2)$-invariant, Poincar\'e-Einstein, conformally K\"ahler metrics on a disk bundle inside the total space of $\mO(-m)$ which do not arise from Derdzi\'nski's theorem but rather are conformal to scalar-flat K\"ahler metrics. We obtain these metrics by taking limits of appropriate rescalings of a degenerating family of our Einstein metrics. We prove the following result in Section \ref{sec:Not Derdzinsky}. 

\begin{thm}\label{thm:Not Derdzinsky}
	For every $m\geq 3$, there is a Poincar\'e-Einstein metric on a disk bundle in $\mO(-m)$ which is not K\"ahler, but conformal to a scalar-flat K\"ahler metric. 
\end{thm}

We point out that our all Einstein metrics are conformally K\"ahler therefore Hermitian. In fact, in dimension $4,$ all Hermitian Einstein metrics are conformally K\"ahler.

The paper is organised as follows. In section \ref{sec:prelim} we give some required background on the toric geometry of $\mO(-m) \rightarrow \cp^1$ and of Hirzebruch surfaces $\bbH_m$. This section also contains a brief discussion on extremal metrics leading to Calabi's ansatz for extremal metrics with a $U(2)$-symmetry and a discussion of Derdzi\'nski's theorem. In section \ref{sec:bachflat} we focus on those extremal metrics arising from the Calabi's ansatz which are Bach-flat and we derive a criterium for Bach-flatness in this setting. In section \ref{sec:local} we use the results from section \ref{sec:bachflat} to give a very general local construction of $U(2)$-invariant Bach-flat K\"ahler metrics on disk bundles in $\mO(-m)\rightarrow \cp^1$ in terms of two parameters, namely the volume of the zero section and the scalar curvature at the zero section. In section \ref{sec:localglobal} we study the maximal space where the constructed metrics are defined and their asymptotics. As it turns out, certain parameter values determine metrics on the total space of $\mO(-m)$ with quartic or exponential volume growth, while other values determine incomplete metrics on the same space. Others values give metrics whose underlying space is a Hirzebruch surface but the metrics generally have a cone angle singularity along the divisor at infinity, which is a $\cp^1.$ In section \ref{sec:hm}, we focus on the case of $\bbH_m$ and we study the possible cone angles arising in our construction as well the vanishing locus for the scalar curvature of our metrics with an eye on Derdzi\'nki's theorem. As it turns out the case of $\bbH_1$ is special and we get metrics whose scalar curvature does not vanish anywhere. The case of $\bbH_2$ is also special in that no Bach-flat metrics exist. In section \ref{sec:cke} we are finally able to apply Derdzi\'nski's construction to get 
\begin{enumerate}
\item Cone-angle Einstein metrics on $\bbH_1.$ We study the possible cone angles we get and confront our results with the Hitchin-Thorpe inequality which was generalised to this setting by Atiyah-Lebrun (see \cite{al}).
\item Poincar\'e-Einstein metrics on the total space of the disk bundle in $\mO(-m)$ which are fillings of lens spaces. Some of these can be capped off by cone angle metrics while other can be capped off by incomplete metrics. The smooth Poincar\'e-Einstein fillings were previously written down by Page and Pope (see \cite{pp}).
\end{enumerate}
In the last section, we write down and study a family of Poincar\'e-Einstein metrics conformal to scalar-flat K\"ahler metrics. 

In Appendix \ref{appendix:convex}  we give a complete proof of a technical nature showing that the symplectic potentials we construct are indeed convex. This is essential for our metrics to be well defined but it is a rather cumbersome calculation.  In Appendix \ref{appendix:mc} we give a classification of the metrics we construct in terms of the parameters values on which they depend. We finish our paper with Appendix \ref{d} where we give more details on Derdzi\'nski's theorem and some proofs for the results we use from \cite{d}.\\

\noindent \textbf{Acknowledgments:} We would like to thank Claude LeBrun for useful conversations and for his interest in this work. We are also grateful to Olivier Biquard for clarifying some facts concerning the Chen-Teo instanton.

\section{Preliminaries}\label{sec:prelim}
\subsection{The total space of $\mO(-m)$} Projective space $\CP^1$ can be thought of as the space of complex lines in $\bbC^2.$ 
$$
\CP^1=\{[z_0:z_1]: (z_0,z_1)\ne (0,0)\},
$$
where for every non-zero vector in $\bbC^2$ $(z_0,z_1),$ $[z_0:z_1]$ denotes the line through $(z_0,z_1)$ minus the zero vector. We see that 
$$
\CP^1=(\bbC^2\setminus \{(0,0)\})/\bbC^*\simeq S^3/S^1
$$
and  $\CP^1$ admits an $S^1$-action. In fact this action preserves the Fubini-study metric and is toric. Its moment map image is an interval in $\bbR$ which can be identified with the dual of the Lie algebra of $S^1.$ Consider the tautological bundle over $\CP^1$ $\mO(-1)$ i.e.
$$
\{([z_0:z_1],(w_0,w_1)):(w_0,w_1) \in [z_0:z_1]\}\rightarrow \CP^1
$$
where the map sends $([z_0:z_1],(w_0,w_1))$ to $[z_0:z_1].$ The total space of this bundle admits an action of $S^1\times S^1=\bbT^2.$ It follows that all powers of this bundle also admit an $S^1\times S^1$-action. Now for each $m$ the total space of $\mO(-m)$ is a symplectic manifold. A way to see this is simply to see the total space of $\mO(-m)$ as a closed subvariety of $\CP^1\times \bbC^2.$ The $\bbT^2=S^1\times S^1$ action preserves the symplectic form and is actually Hamiltonian for all integers $m$ so that the total space of $\mO(-m)$ is non-compact and toric. The image of the moment map of the $\bbT^2$-action is a non-compact convex polytope $P$
whose interior normals are 
$$
(0,1),(1,0), (m,-1).
$$
This polytope is of Delzant type, in particular the determinant of interior normals corresponding to adjacent facet in indirect order is $-1.$ Over the pre-image of the interior of the polytope, the $\bbT^2$-action is free and that pre-image can be identified with the interior of the polytope times $\bbT^2.$ This identification yields coordinates which are called action-angle coordinates.

There is a linear transformation $T_m$ taking this polytope to the polytope whose normals are  
$$
(0,1),(1/m,1/m), (0,1).
$$
This polytope is not Delzant but it will prove useful in what we do next. We denote it by $P_\mO.$
$$
P_\mO=\{(x_1,x_2)\in \bbR^2: x_1,x_2\geq 0, \, a\leq x_1+x_2\}.
$$
Over the pre-image of the interior of the moment polytope via the moment map, $(x_1,x_2,\theta_1,\theta_2),$ where $(\theta_1,\theta_2)$ are angle coordinates on $\bbT^2,$ are coordinates on our manifold. 

\subsection{Hirzebruch surfaces} Let $\bbH_m$ denote the $m$-th Hirzebruch surface, that is $\bbH_m=\mathbb{P}(\mathcal{O}(-m)\oplus \mathcal{O})\rightarrow \cp^1$ the total space of the projectivized bundle $\mathcal{O}(-m)\oplus \mathcal{O}\rightarrow \cp^1.$ This admits a $\bbT^2=S^1\times S^1$-action. When we choose a compatible symplectic form on $\bbH_m$ this action becomes toric i.e. it admits a moment map determined by the action and the symplectic form. The image of the moment map of each symplectic Hirzebruch surface has 4 facets and the corresponding interior normal are
$$
(0,1),(1,0),(m,-1),(-1,0).
$$
Again this polytope is of Delzant type and the pre-image via the moment map of the interior of this polytope is the open dense set in $\bbH_m$ where the $\bbT^2$-action is free. We again get action-angle coordinates on that set. There is an element in $GL(2,\bbR)$ inducing a linear transformation $T_m$ taking this to the polytope whose normals are 
$$
(0,1),(1/m,1/m),(0,1),(-1/m,1/m),
$$
whose matrix representation is
$$
\begin{pmatrix}
m&-1\\
0&1\\
\end{pmatrix}. 
$$
We denote this open polytope by $P_H.$ 
$$
P_H=\{(x_1,x_2)\in \bbR^2: x_1,x_2\geq 0, \, a\leq x_1+x_2 \leq b\}.
$$
It is not of Delzant type (as the transformation is not actually in $SL(2,\bbZ)$) but it will be convenient for the calculations. The constants $a$ and $b$ reflect the choice of a cohomology class on the Hirzebruch surface to which the symplectic form belongs i.e. a polarisation. The transformation $T_m$ composed with the moment map for the $\bbT^2$-action on $\bbH_m$ fibres over the interior of the moment polytope as a $\bbT^2$ fibration. Over the pre-image of the interior of the moment polytope, $(x_1,x_2,\theta_1,\theta_2)$ where $(\theta_1,\theta_2)$ are angle coordinates on $\bbT^2$ are coordinates on our manifold. The pre-image of each facet of the moment polytope is a $\cp^1.$ For each $\alpha$, we will denote the pre-image of $$\{(x_1,x_2): x_1>0, x_2>0, \alpha=x_1+x_2\}$$ via $T_m$ composed with the moment map by $S_\alpha.$ When $\alpha=b$ this pre-image is actually a $\CP^1$ and we sometimes refer to it as the divisor at infinity.

This description for Hirzebruch surfaces was already considered in \cite{acoho}.
\subsection{Symplectic potential} We are going to be interested in considering metrics in the total space of $\mO(-m)$ and on $\bbH_m$ which are compatible with the symplectic structure on these manifolds i.e. K\"ahler metrics for which the $\bbT^2$-action is biholomorphic and Hamiltonian. Such metrics are called toric K\"ahler metrics and have been studied in greater generality. We refer the reader to \cite{a0} for background. 

For our purposes it is enough to know that toric K\"ahler metrics on the total space of $\mO(-m)$ and on $\bbH_m$ are parametrised by certain convex functions on their moment polytopes with prescribed singular behaviour on the boundary. For an explanation of this fact see \cite{a0} and \cite{aorb}.

Using the action-angle setup and the transformation $T_m$, any toric K\"ahler metric on  total space of $\mO(-m)$ or on $\bbH_m$ is determined by a function $u$ which 
\begin{itemize}
\item \label{cond_convex} is convex either on the interior of $P_\mO$ or of $P_H$, i.e. $\Hess(u)$ is positive definite there,
\item in the case of the total space of $\mO(-m),$
\[
\scalebox{.7}{$ u-\frac{1}{2}\left(x_1\log(x_1)+x_2\log(x_2)+\frac{x_1+x_2-a}{m}\log(x_1+x_2-a) \right)$}
\]
is smooth and satisfies
$$
\det(\Hess(u))=\frac{\delta}{x_1x_2(x_1+x_2-a)},
$$
for a positive, smooth function $\delta$ on $P_\mO$
\item in the case of $\bbH_m,$
\[
\scalebox{.7}{$ u-\frac{1}{2}\left(x_1\log(x_1)+x_2\log(x_2)+\frac{x_1+x_2-a}{m}\log(x_1+x_2-a)+\frac{b-(x_1+x_2)}{m}\log(b-(x_1+x_2))\right) $}
\]
is smooth and satisfies
$$
\det(\Hess(u))=\frac{\delta}{x_1x_2(x_1+x_2-a)(b-(x_1+x_2))},
$$
for a positive, smooth function $\delta$ on $P_H$
\end{itemize}

Conversely, in both cases all such functions $u$ satisfying the described conditions determine a metric which we denote by $g_u$
\begin{equation}\label{g_u}
g_u=\sum_{i,j=1}^2 \left( u_{ij} dx_i \otimes dx_j + u^{ij}d\theta_i\otimes d\theta_j \right)
\end{equation}
where $u_{ij}$ denote the entries of $\Hess(u)$, $u^{ij}$ those of its inverse, and $(\theta_1,\theta_2)$ are the fibre coordinates on the surfaces. 

We will consider metrics on the Hirzebruch surface with a cone angle along the $\cp^1$ whose image via $T_m$ composed with the moment map  is $S_b.$ The boundary behaviour of the corresponding symplectic potential is determined by the fact that 
$$
\scalebox{.9}{$u-\frac{1}{2}\left(x_1\log(x_1)+x_2\log(x_2)+\frac{x_1+x_2-a}{m}\log(x_1+x_2-a)+\frac{p(b-(x_1+x_2))}{m}\log(b-(x_1+x_2	))\right)$},
$$
is smooth for a positive number $p$ given by $2\pi$ divided by the cone angle. We refer to $p$ as a weight. 

\subsection{Abreu's formula and Calabi's ansatz for extremal metrics}
Abreu showed that the scalar curvature of $g_u$ can be computed in terms of $u:$
$$
\scal(g_u)=-\left( \frac{\partial^2 u^{11}}{\partial x_1^2}+2\frac{\partial^2 u^{12}}{\partial x_1\partial x_2}+\frac{\partial^2 u^{22}}{\partial x_2^2}\right).
$$
This function depends only on $x_1$ and $x_2$ as the metric is torus invariant. The formula above is known as Abreu's formula.\\

In K\"ahler geometry a central role is played by extremal metrics. These were introduced by Calabi as critical points of the Calabi functional on a K\"ahler class. More specifically given a fixed K\"ahler metric $\omega_0$ on a K\"ahler manifold $M,$ its K\"ahler class, which is determined by the cohomology of $\omega_0$ is given by
$$
\{\omega_0+i\partial\bar{\partial}\varphi: \quad \omega_0+i\partial\bar{\partial}\varphi>0\}.
$$
The Calabi functional on the other hand is defined by
$$
\text{Cal}(\omega)=\int_M\left|\scal(g_\omega)\right|^2,
$$
where $g_\omega$ is the Riemannian metric associated with the K\"ahler form $\omega$ and the fixed complex structure on $M,$ and $\scal$ denotes the scalar curvature. Constant scalar curvature K\"ahler metrics are always extremal in the sense of Calabi but there extremal metrics which have non-constant scalar curvature.

As it turns out, a toric K\"ahler metric $g_u$ is extremal in the sense of Calabi iff $\scal(g_u)$ is an affine function of $(x_1,x_2).$ See \cite{a0} for more details.

Calabi also developed an ansatz he was able to use to construct several examples of extremal metrics. Calabi's ansatz yields extremal metrics on the total space of some special bundles over K\"ahler-Einstein manifolds. We are going to use it for the bundles $\mO(-m)\rightarrow \cp^1$ when $m>0$.  Both the total space of $\mO(-m)$ and all Hirzebruch surfaces admit a $U(2)$-action and Calabi's metrics are $U(2)$-invariant.  Calabi famously applied his ansatz to explicitly construct extremal non-constant scalar curvature examples of metrics on each $\bbH_m$. His methods were translated into the language of action-angle coordinates by Raza \cite{r} and Abreu \cite{acoho}. In particular Abreu explained that the polytopes $P_\mO$ and $P_H$ are well suited to describe any $U(2)$-invariant metric. The symplectic potential of such a $U(2)$-invariant K\"ahler metric can be written as
\begin{equation}\label{eq:symplectic potential}
	u_h=\frac{1}{2}\left(x_1\log(x_1)+x_2\log(x_2)+h(x_1+x_2)\right),
\end{equation}
where 
\begin{enumerate}
\item \label{smoothness_in_h_Om} In the case of $\mO(-m)$ the function $h$ is smooth in $]a,+\infty[$ and is such that
$$
h''(r)-\frac{1}{m(r-a)}
$$
is smooth on $[a,+\infty[,$ and $(x_1,x_2)$ take values in $P_\mO.$ 

\item \label{smoothness_in_h_Hm} In the case of $\bbH_m, $ the function $h$ is smooth in $]a,b[$ and is such that
$$
h''(r)-\frac{1}{m}\left(\frac{1}{r-a}+\frac{p}{b-r}\right)
$$
is smooth on $[a,b],$ and $(x_1,x_2)$ take values in $P_H.$
\end{enumerate}

Before we explain Calabi's ansatz we start by noticing that the above boundary conditions ensure that $u_h,$ as defined by equation (\ref{eq:symplectic potential}), satisfies the right boundary conditions to be a symplectic potential. 
\begin{lemma}\label{lemma:smooth}
	Let $h$ be a function satisfying either conditions in (\ref{smoothness_in_h_Om}) or  (\ref{smoothness_in_h_Hm}) above. Let $u=u_h$ be determined by $h$ via equation (\ref{eq:symplectic potential}). Then either
	\begin{itemize}
	\item In the case of $\mO(-m)$, 
	$$
        x_1x_2(x_1+x_2-a)\det(\Hess(u)),
        $$
        is positive and smooth on $P_\mO.$
        \item In the case of $\bbH_m$, 
	$$
        x_1x_2(x_1+x_2-a)(b-(x_1+x_2))\det(\Hess(u)),
        $$
        is positive and smooth on $P_H.$
        \end{itemize}
\end{lemma}
\begin{proof}
We have $$
	\Hess(u)=
	\frac{1}{2}\begin{pmatrix}
		\frac{1}{x_1} + h''&  h''\\
		h''&\frac{1}{x_2}+ h''\\
	\end{pmatrix},
	$$
	so that 
	$$
\det(\Hess(u))=\frac{1+rh''}{x_1x_2}.
$$
When either condition (\ref{smoothness_in_h_Om}) or (\ref{smoothness_in_h_Hm}) are satisfied this ensures that either
$$
{x_1x_2(x_1+x_2-a)}\det(\Hess(u))=(r-a)(1+rh''),
$$
is positive and smooth on the closure of $P_\mO$ or
$$
{x_1x_2(x_1+x_2-a)(b-(x_1+x_2))}\det(\Hess(u))=(r-a)(b-r)(1+rh''),
$$
is a positive and smooth on the closure of $P_H.$
\end{proof}

We quickly review Calabi's method in the toric framework. We have said that a toric metric $g_u$ is extremal if and only if $\scal(g_{u})$ is an affine function of $(x_1,x_2).$ Now taking into account that $u=u_h$ is given by equation \ref{eq:symplectic potential} we can apply Abreu's formula for the scalar curvature to deduce that $h$ must be of a very special form in order for $g_{u_h}$ to be extremal.

\begin{thm}[Calabi's ansatz]\label{prop:scalar curvature and extremal metrics}\indent
	Let $h$ be a function satisfying either conditions in (\ref{smoothness_in_h_Om}) or  (\ref{smoothness_in_h_Hm}) above. Let $u=u_h$ be determined by $h$ via equation (\ref{eq:symplectic potential}). Assume that $\Hess(u_h)$ is positive definite. Let also $q$ be such that
		$$h''(r)= \frac{q(r)}{r(r^2-q(r))}.$$  The scalar curvature of $g_{u_h}$, the toric K\"ahler metric induced by the symplectic potential $u_h$, is given by
	$$\scal(g_{u_h})=\frac{2q''(r)}{r},$$ wherever $g_{u_h}$ is defined.
	
	In particular, $g_{u_h}$ is extremal if and only if $q$ is a quartic polynomial of the form $q(r)=\sum_{i=1}^4 \frac{q_i}{i!} r^i$ with $q_2=0$, in which case the scalar curvature is given by
	$$\scal(r)=2q_3 + q_4 r.$$
\end{thm}
\begin{proof}
	We follow \cite{aorb}. The proof uses Abreu's formula
	\begin{align*}
		\scal(g_{u_h})= - \frac{\partial^2 u^{ij} }{\partial x_i \partial x^j}.
	\end{align*}
	Now
	$$
	\Hess(u_h)=
	\frac{1}{2}\begin{pmatrix}
		\frac{1}{x_1} + h''&  h''\\
		h''&\frac{1}{x_2}+ h''\\
	\end{pmatrix}.
	$$
	Set $f$ to be such that $h''=\frac{f}{1-rf}$ and rewrite 
	$$
	\Hess(u_h)=
	\frac{1}{2}\begin{pmatrix}
		\frac{1}{x_1} + \frac{f}{1-rf}&  \frac{f}{1-rf}\\
		\frac{f}{1-rf} &\frac{1}{x_2}+ \frac{f}{1-rf}\\
	\end{pmatrix}.
	$$
	Then, we find that
	$$
	(\Hess(u_h))^{-1}=2
	\begin{pmatrix}
		x_1-x_1^2f&-x_1x_2 f\\
		-x_1x_2 f&x_2-x_2^2f\
	\end{pmatrix},
	$$

	and inserting the formula for $u^{ij}$ in terms of $f$ gives the equation 
	$$\scal(g_u)=2r^2 f''(r)+12 r f'(r) +12 f(r).$$
	Set $f(r)=q(r)r^{-3}.$ Replacing in the previous formula yields the formula for $\scal(g_u)$ in terms of $q.$
	
	Next, the metric $g_{u_h}$ is extremal if and only if its scalar curvature is an affine function of $x_1,x_2$. However, as $\scal(g_u)=\frac{2q''}{r}$ is a function of $r$, the only possibility is that it be affine in $x_1+x_2=r.$ Say it is $A_0 + A_1r$. The resulting equation is $2q''=A_0r+A_1r^2$ and integrating this gives
	$$q(r)= q_0 + q_1 r + \frac{A_0}{12} r^3 + \frac{A_1}{24} r^4,$$
	i.e. $q_2=0$, $q_3=\frac{A_0}{2}$, and $q_4=A_1$. Back to the scalar curvature, it can then be written as
	$$\scal(r)=2q_3 + q_4 r,$$
	and we are done.
	
\end{proof}
In summary a $U(2)$-invariant, K\"ahler metric either on the total space of $\mO(-m)$ or on $\bbH_m$ is determined by a one variable function. For future reference we leave a summary of the encoding functions we have used. 
\begin{remark}[On notation]\label{rem:Different parametrizations}
	A  $U(2)$-invariant, K\"ahler metric on the total space of $\mO(-m)$ or on $\bbH_m$ is determined by either:
	\begin{itemize}
	        \item A function $h$ determining a symplectic potential $u_h$ via equation (\ref{eq:symplectic potential}).
		\item A function $f$ of the form
		$$h''(r)=\frac{f(r)}{1-rf(r)} .$$
		\item A function $q$ such that
		$$h''(r)= \frac{1}{r}\frac{q(r)}{r^2-q(r)},$$
		in which case $f(r)$ and $q(r)$ are related by $f(r)=\frac{q(r)}{r^3}$.
		\item A function $\pol$ such that
		$$h''(r)= -\frac{1}{r} + \frac{r}{\pol(r)},$$
		and so $\pol(r)=r^2-q(r)$. Sometimes it will also be convenient to consider another way to parameterise the geometry which is scale invariant. For this we use $t=\frac{r}{a}$ and the polynomial $\tilde{\pol}(t)=\frac{\pol(at)}{a^2}$.
	\end{itemize} 
	A  $U(2)$-invariant, K\"ahler, {\bf extremal} metric on the total space of $\mO(-m)$ or on $\bbH_m$ is determined by either:
	\begin{itemize}
	        \item A function $h$ satisfying
	        $$h''(r)= -\frac{1}{r} + \frac{r}{\pol(r)},$$
	        for a degree $4$ polynomial $\pol$ whose quadratic term is $r^2$.
		\item A function $f$ such that $f(r)=\frac{q(r)}{r^3}$ for a degree $4$ polynomial $q$ whose quadratic term is zero. 
		\item A degree $4$ polynomial $q$ whose quadratic term is zero.
		\item A degree $4$ polynomial $\pol(r)$ whose quadratic term is $r^2$ such that
		 $\pol(r)=r^2-q(r)$.
	\end{itemize} 
\end{remark}

\subsection{Smoothness and convexity}

\subsubsection{Smoothness}

The smoothness of the metric $g_{u_h}$ follows from the boundary conditions on $u_h$ and it follows from Lemma \ref{lemma:smooth} that the conditions on $u_h$ will be met as long as either condition (\ref{smoothness_in_h_Om}) or (\ref{smoothness_in_h_Hm}) are satisfied. These can be summed up in terms of the notation above as
\begin{itemize}
\item  In the case of the total space of $\mO(-m)$
\begin{equation}\label{eq:conditions for closing smoothly in terms of q1}
	\begin{aligned} 
		q(a) = a^2 & ,  \\ 
		q'(a)= (2-m) a&  . 
	\end{aligned}
	\end{equation}
\item In the case of $\bbH_m$
\begin{equation}\label{eq:conditions for closing smoothly in terms of q2}
	\begin{aligned} 
		q(a) = a^2 & , \quad q(b)=b^2 \\ 
		q'(a)= (2-m) a & , \quad q'(b)= \left( 2+ \frac{m}{p} \right) b . 
	\end{aligned}
\end{equation}
\end{itemize}

Similarly, they can be rewritten in terms of $\pol(r)$ as 
\begin{itemize}
\item In the case of the total space of $\mO(-m)$
\begin{equation}\label{eq:conditions for closing smoothly in terms of p1}
	\begin{aligned} 
	&\pol(a)=0,  \\ 
	&\pol'(a)=am.
	\end{aligned}
\end{equation}
\item In the case of $\bbH_m$
\begin{equation}\label{eq:conditions for closing smoothly in terms of p2}
	\begin{aligned} 
	&\pol(a)=0 = \pol(b), \\ 
	&\pol'(a)=am, \quad \pol'(b)=\frac{-bm}{p}.
	\end{aligned}
\end{equation}
\end{itemize}

\subsubsection{Convexity} 
In this subsection we will derive a simple criterium to check convexity for symplectic potentials of $U(2)$-invariant K\"ahler metric i.e. metrics whose symplectic potentials are given in terms of a function $h$ via equation (\ref{eq:symplectic potential}). 
\begin{lemma}\label{lem:Convexity}
	{Let $h$ be a smooth function in $]a,+\infty[$ or in $]a,b[$ and let $u$ be defined by $h$ via equation (\ref{eq:symplectic potential}).} Then
	$\Hess(u)$ is positive definite on the interior of $P_\mO$ or of $P_H$ respectively if and only if the {function}
	$$\pol(r):=r^2-q(r)$$
	is positive on $[a,+\infty[$ or $[a,b]$ respectively.

\end{lemma}
\begin{proof}
	It turns out to be a little more convenient to prove that $(\Hess(u))^{-1}$ is positive definite. To do so we compute $(\Hess(u))^{-1}$ as follows. First, we know that
	$$
	\Hess(u)=
	\frac{1}{2}\begin{pmatrix}
		\frac{1}{x_1} + h''&  h''\\
		h''&\frac{1}{x_2}+ h''\\
	\end{pmatrix},
	$$
	and substituting $h''=\frac{f}{1-rf}$ gives
	$$
	\Hess(u)=
	\frac{1}{2}\begin{pmatrix}
		\frac{1}{x_1} + \frac{f}{1-rf}&  \frac{f}{1-rf}\\
		\frac{f}{1-rf} &\frac{1}{x_2}+ \frac{f}{1-rf}\\
	\end{pmatrix},
	$$
	Then, we find that
	$$
	(\Hess(u))^{-1}=2
	\begin{pmatrix}
		x_1-x_1^2f&-x_1x_2 f\\
		-x_1x_2 f&x_2-x_2^2f\
	\end{pmatrix},
	$$
	and so
	\begin{align*}
		\mathrm{tr} ((\Hess(u))^{-1}) & = 2 (x_1-x_1^2f+ x_2-x_2^2f) \\
		& = 2 (r - (r^2-2x_1x_2)f ) \\
		& = 4x_1x_2 f + 2r (1-rf) \\
		\det(\Hess (u)) & =  \frac{1}{4x_1 x_2(1-rf)}.
	\end{align*}
	These can be rewritten in terms of $q$ as
	\begin{align*}
		\mathrm{tr} ((\Hess(u))^{-1}) & =  \frac{4x_1x_2q}{r^3} +  \frac{2(r^2-q)}{r} \\
		\det(\Hess (u)) & = \frac{r^2}{4x_1 x_2(r^2-q)}.
	\end{align*}
	Now, the metric is positive definite if and only if $\mathrm{tr} ((\Hess(u))^{-1})$ and $\det(\Hess (u))$ are both positive. From the second we find that $r^2-q$ must be positive. From the first, we have $\frac{2x_1x_2q}{r^2}  +r^2 -q>0$. If $q>0$, then it is enough that $r^2-p>0$. On the other hand, given that $0 \leq x_1x_2 \leq \frac{r^2}{4}$, if $q<0$ we have $x_1x_2 q \geq {\frac{r^2q}{4}}$ and so
	$$\frac{2x_1x_2q}{r^2}  +r^2 -q > \frac{q}{2} + r^2 -q = - \frac{q}{2} + r^2 >0,$$
	as both terms are positive.

\end{proof}

\subsection{Einstein metrics from conformally K\"ahler geometry} 

 In differential geometry one is often looking for a canonical metric which reflects intrinsic properties of an underlying manifold. Extremal K\"ahler metrics can play that role but it is known that these do not always exist. Also, when they exist, extremal metrics are the ``best" in their K\"ahler class. It is then natural to look for other special metrics in the conformal class of a K\"ahler metric. 
Derdzi\'nski considered the problem of finding conformally K\"ahler Einstein metrics (see \cite{d}).
\begin{question*}
Let $M$ be a K\"ahler manifold. For which K\"ahler metrics $g_K$ on $M$ does there exist a function $\sigma$ such that $e^\sigma g_K$ is Einstein? 
\end{question*}
In \cite{d} Derdzi\'nski answers the question by proving the following theorem.
\begin{thm}[Derdzi\'nski]\label{Thmderd1}
Let $(M,g_K,\omega)$ be a K\"ahler manifold whose scalar curvature in nowhere vanishing. If the Hermitian metric $\frac{g_K}{\scal^2(g_K)}$ is Einstein then $g_K$ is Bach-flat and extremal. In this setting, if $g_K$ extremal then $g_K$ is Bach-flat if and only if the scalar curvature of $\frac{g_K}{\scal^2(g_K)}$ is constant.
\end{thm}
The Bach tensor is a conformal invariant defined using the Riemann curvature tensor. For more on the Bach tensor see Appendix \ref{d}. In \cite{d} there is also a converse to the above theorem. Namely
\begin{thm}[Derdzi\'nski]\label{Thmderd2}
Let $(M,g_K,\omega)$ be an Bach-flat K\"ahler manifold. Then, the Hermitian metric $\frac{g_K}{\scal^2(g_K)}$ is Einstein wherever defined.
\end{thm}
We say more on Derdzi\'nski's work in Appendix \ref{d}. In particular we discuss its proof. In fact, in dimension $4,$ K\"ahler metrics which are Bach-flat are always extremal (see \cite{l2}), so even though we are apparently using the extremality condition, in reality we are not.

\section{Extremal and Bach-flat K\"ahler metrics}

\subsection{Bach-flateness}\label{sec:bachflat}

Our goal in this paper is to find Einstein metrics which are $U(2)$-invariant and conformally K\"ahler. In view of Derdzi\'nski's Theorems we must start by looking for Bach-flat metrics. Since we have a large family of extremal metrics from Calabi's ansatz as in Theorem \ref{prop:scalar curvature and extremal metrics} we shall move to understand which of those extremal metrics are Bach-flat.

In our setting, Theorem \ref{Thmderd1} leads to the following.
\begin{prop}\label{prop:Conformal Einstein}
	{Consider the polynomial $$q(r)=q_0+q_1r+\frac{q_3}{6}r^3+\frac{q_4}{12}r^4.$$} Let $h$ be such that $$h''(r)= \frac{q(r)}{r(r^2-q(r))}.$$
	Let $g=g_{u_h}$ be the extremal K\"ahler metric determined by $h$ as in Proposition \ref{prop:scalar curvature and extremal metrics}. Such a metric is Bach-flat if and only if
	\begin{equation}\label{eq:Bach flat}
		q_3 q_1  = q_4 q_0.
	\end{equation}
	Equivalently, where defined, the conformally related metric $(\scal(g))^{-2}g$ is Einstein if and only if equation (\ref{eq:Bach flat}) holds. In this situation, its scalar curvature is given by
	\begin{align*}
		S & = 12 q_4^2 q_1 + 8 q_3^3 + 48 q_3 q_4.
	\end{align*}
\end{prop}

Before presenting the proof of this proposition we start by proving an auxiliary result.

\begin{lemma}\label{lemma:Laplacian gamma}
	Let $g_u$ be a toric K\"ahler metric induced by a symplectic potential as in (\ref{eq:symplectic potential}) and let $\gamma$ be a function of a single variable. Consider $\gamma (r) = \gamma(x_1+x_2)$ as a function on $\bbH_m.$ Then, the Laplacian $\Delta \gamma$ of this function with respect to $g_u$ is given in terms of $f$ by
	\begin{equation*}
		-\frac{1}{2}\Delta \gamma =\left(2-3rf-r^2f'\right)\gamma'+\left(r-r^2f\right)\gamma''
	\end{equation*}
	or in terms of $q$ by
	\begin{equation*}
		-\frac{1}{2}\Delta \gamma =\left(2-\frac{q'}{r}\right)\gamma'+\left(r- \frac{q}{r}\right)\gamma''.
	\end{equation*}
where $f$ and $q$ are as in Remark \ref{rem:Different parametrizations}.
\end{lemma}
\begin{proof}
	In action-angle coordinates we can write the Laplacian of a torus invariant function $\gamma$ as
	$$
	\Delta \gamma=-\frac{\del}{\del x_1}\left(u^{11}\frac{\del \gamma}{\del x_1}+u^{12}\frac{\del \gamma}{\del x_2}\right)-\frac{\del}{\del x_2}\left(u^{12}\frac{\del \gamma}{\del x_1}+u^{22}\frac{\del \gamma}{\del x_2}\right).
	$$
	Let $\gamma$ be any function depending only on $r=x_1+x_2$. Then $$\frac{\del \gamma}{\del x_1}(r)=\frac{\del \gamma}{\del x_2}(r)=\gamma'(r)$$ so
	$$
	\Delta \gamma=-\frac{\del}{\del x_1}\left((u^{11}+u^{12})\gamma'\right)-\frac{\del}{\del x_2}\left((u^{12}+u^{22})\gamma'\right).
	$$
	Now, recall that
	$$
	(\Hess(u))^{-1}=2
	\begin{pmatrix}
		x_1-x_1^2f&-x_1x_2 f\\
		-x_1x_2 f&x_2-x_2^2f\
	\end{pmatrix}.
	$$
	Inserting this into the formula for the Laplacian we see that
	$$
	\begin{aligned}
		-\frac{1}{2}\Delta \gamma&=\frac{\del}{\del x_1}\left((x_1-x_1^2f-x_1x_2 f)\gamma'\right) + \frac{\del}{\del x_2}\left((-x_1x_2 f+x_2-x_2^2f)\gamma'\right)\\
		&=\left(1-2x_1 f-x_2 f-(x_1^2+x_1x_2)f'-x_1 f+1-2x_2 f-(x_1 x_2+x_2^2)f'\right)\gamma' \\
		& \ \ \ \ +\left(x_1-x_1^2f-x_1 x_2 f-x_1 x_2 f+x_2-x_2^2 f\right)\gamma''\\
		&=\left(2-3rf-r^2f'\right)\gamma'+\left(r-r^2f\right)\gamma'',
	\end{aligned}
	$$
	as we wanted to prove. The formula in terms of $q$ follows from direct substitution.
\end{proof}

\begin{proof}[Proof of Proposition \ref{prop:Conformal Einstein}]
	It follows from Derdzi\'nski's Theorem \ref{Thmderd1} that the extremal toric K\"ahler metric $g$ is Bach-flat if and only if $(\scal(g))^{-2}g$ is Einstein and, in this situation, this happens if and only if the scalar curvature $S$ of $(\scal(g))^{-2}g$ is constant. So, we need only compute $S$ which is given by
	$$S= (\scal(g))^{3}( 6 \Delta (\scal(g))^{-1} + 1 ).$$
	Now, we know that $\scal(g) = A_0 + A_1 r$ with $A_0=2q_3$ and $A_1=q_4$.  Setting $\gamma=(\scal(g))^{-1}$ we find that $\gamma'(r)=-A_1(A_1r+A_0)^{-2}=-A_1(\scal(g))^{-2}$ and $\gamma''(r)=2A_1^2(A_1r+A_0)^{-3}=2A_1^2 (\scal(g))^{-3}$. Hence, using lemma \ref{lemma:Laplacian gamma}
	\begin{align*}
		(\scal(g))^3 \Lap (\scal(g))^{-1} & = -2 \left(2-\frac{q'}{r}\right) (\scal(g))^3 \gamma'-2\left(r- \frac{q}{r}\right) (\scal(g))^3\gamma'' \\
		& = 2 \left(2-\frac{q'}{r}\right) A_1 \scal(g) -2\left(r- \frac{q}{r}\right) 2A_1^2.
	\end{align*} 
	and the scalar curvature of $(\scal(g))^{-2}g_{u}$ is therefore given by
	\begin{align*}
		S & = (\scal(g))^{3}( 6 \Lap (\scal(g))^{-1} + 1 ) \\
		& =  12 \left(2-\frac{q'}{r}\right) A_1 \scal(g) - 24\left(r- \frac{q}{r}\right) A_1^2 + (\scal(g))^3 .
	\end{align*}
	Inserting $\scal(g)=A_0 + A_1r$ we further obtain
	$$S= \frac{r (A_0+A_1r)^3 + 24 A_1^2 q  - 12 A_1 (A_0+A_1r) q'   }{r} + 24 A_1 A_0,$$
	and this is constant if and only if 
	$$r (A_0+A_1r)^3 + 24 A_1^2 q  - 12 A_1 (A_0+A_1r) q'    = \lambda r,$$ 
	for some constant $\lambda \in \R$.
	
	At this point we insert the degree $4$-polynomial $q(r)=\sum_{i=1}^4 \frac{q_i}{i!} r^i$ into the equation above. The left hand side of the equation above becomes
	\begin{align*}
		2 (12A_1^2 q_0 -6A_0 q_1) + 2(6A_1^2 q_1 +A_0^2 q_3) r + ( A_0^2 q_4 -2A_1 A_0 q_3  ) r^2.
	\end{align*}
	Recalling that $q_2=0$, the requirement that the terms of order different than $O(r)$ vanish results in the following two equations
	\begin{align*}
		12A_1^2 q_0 -6A_1 A_0 q_1 & = 0 \\
		-2A_1A_0q_3 +A_0^2q_4 & = 0,
	\end{align*}
	which for nonzero $A_1,A_0$ (i.e. nonzero $q_4,q_3$) we can rewrite as
	\begin{align*}
		A_0 q_1 & = 2A_1 q_0 \\
		A_0 q_4 & = 2A_1 q_3.
	\end{align*}
	Then, using the fact that $q_3=\frac{A_0}{2}$ and $q_4=A_1$ these become 
	\begin{align*}
		q_3 q_1 & = q_4 q_0 \\
		q_3 q_4 & = q_4 q_3.
	\end{align*}
	The second of these is vacuous, leaving us with the single equation in the statement.
	
	Furthermore, in this situation the scalar curvature $S$ of $(\scal(g))^{-2}g$ can be written by replacing these relations into the above formula for $S$. This gives
	\begin{align*}
		S & = 2(6A_1^2 q_1 +A_0^2 q_3)  + 24 A_1 A_0 \\
		& = 12 q_4^2 q_1 + 8 q_3^3 + 48 q_3 q_4.
	\end{align*}
\end{proof}

\subsection{Local existence}\label{sec:local}
This subsection proves the existence of a $2$-parameter family (including scaling) of Bach-flat K\"ahler metrics defined around the zero section in the total space of $\mO(-m)$. By Derdzi\'nski's theorem these give rise to a locally defined $2$-parameter family of Einstein metrics ($1$-parameter if we factor out scaling).

\begin{thm}\label{thm:Local existence and uniqueness}
	Let $m \in \N$, $a>0$ and $s(a) \in \R$. Then, there is a unique $U(2)$-invariant Bach-flat K\"ahler metric on a disk bundle inside $\mO(-m)$ such that the volume of the zero section is $a$ and the scalar curvature at the zero section is $s(a)$.
\end{thm}
\begin{proof}
	{This is very much like the proof of existence for extremal metrics on $\bbH_m.$ Just as in that setting, such an extremal metric comes from a polynomial 
	$$q(r)=q_0+q_1r+\frac{q_3}{6}r^3+\frac{q_4}{12}r^4.$$
	Let $g$ be the extremal K\"ahler metric determined by $q$ as in Proposition \ref{prop:scalar curvature and extremal metrics}.}
	We need not impose the smoothness condition at $x_1+x_2=b$ and instead impose it only  at $x_1+x_2=a$ as in condition (\ref{eq:conditions for closing smoothly in terms of q1}). We have that
	
	\begin{equation}\label{eq:conditions for closing smoothly 2}
		\begin{aligned} 
			q(a) & = a^2 \\ 
			q'(a) &= (2-m) a  . 
		\end{aligned}
	\end{equation}
	
	which can be easily solved for $q_0$, $q_1$ in terms of the remaining coefficients as
	
	\begin{equation}\label{eq:smoothness conditions on q_0 and q_1}
		\begin{aligned}
			q_{{0}} & ={\frac {q_{{4}}{a}^{4}}{8}}+{\frac {q_{{3}}{a}^{3}}{3}}+{a}^{2
			}(m-1) \\
			q_{{1}} & =-{\frac {q_{{4}}{a}^{3}}{6}}-{\frac {q_{{3}}{a}^{2}}{2}}-a(m-2).
		\end{aligned}
	\end{equation}
	
	Furthermore, imposing the Bach-flatness condition, $q_3 q_1  = q_4 q_0$ yields the equation
	$${q_{{4}}}^{2}{a}^{3}+4\,q_{{4}}q_{{3}}{a}^{2}+4\,a{q_{{3}}}^{2}+8
	\left( m-1 \right) q_{{4}}a+8\,q_{{3}} \left( m-2 \right) =0.
	$$
	By completing the square, this can be written as
	$$ a ( aq_4 + 2q_3 )^2 - 8(aq_4+2q_3) + 8 m( aq_{{4}} + q_{{3}}) =0.
	$$
	Notice that $aq_4+2q_3=s(a)$ is the scalar curvature at the $\T^2$-invariant $2$-sphere located at $r=a$. We solve this equation for $aq_4+q_3$ in terms of $s(a)$. This reads
	\begin{align*}
		aq_{{4}} + q_{{3}} &  =  s(a) \frac{1- \frac{a s(a)}{8}}{m} \\
		aq_{{4}} + 2q_{{3}} &  =  s(a) ,
	\end{align*}
	and can be inverted to
	\begin{equation}\label{eq:q_3 q_4 from s(a)}
		\begin{aligned}
			q_{{3}} &  =  s(a) \frac{a s(a) +8 (m-1) }{8m} \\
			q_{{4}} &  =  -\frac{s(a)}{a} \frac{a s(a) +4 (m-2) }{4m} .
		\end{aligned}
	\end{equation}
	Now, we insert equation (\ref{eq:q_3 q_4 from s(a)}) for $q_3$ and $q_4$ in terms of $a$ and $s(a)$, into equation (\ref{eq:smoothness conditions on q_0 and q_1}) for $q_0$ and $q_1$ in terms of $q_3,q_4$. We find that:
	\begin{equation}\label{eq:q's for local metrics}
	\begin{aligned}
		q_0 & = \frac {{a}^{2} \left( as(a)+8\,m-8 \right)  \left( as(a)+12\,m \right) }{96
			\,m}, \\
		q_1 & = -\frac {a \left( as(a)+4\,m-8 \right)  \left( as(a)+12\,m \right) }{48\,m}, \\
		q_2 & = 0, \\
		q_3 & = \frac {s(a) \left( as(a)+8\,m-8 \right) }{8\,m}, \\
		q_4 & = -\frac {s(a) \left( as(a)+4\,m-8 \right) }{4\,am}.
	\end{aligned}
	\end{equation}
	Using these formulae, we shall consider the polynomial $\frac{r^2-q(r)}{a^2}$ which we write in terms of $t=\frac{r}{a}$ as $\tilde{\pol}(t)=\frac{(at)^2-q(at)}{a^2}$. This results in the polynomial
	\begin{align*}
		\tilde{\pol}(t) & =-{\frac { \left( as(a)+8\,m-8 \right)  \left( a s(a)+12\,m \right) }{96\,m}}
		+{\frac { \left( a s(a) +4\,m-8 \right)  \left( a s(a)+12\,m \right) }{48\,m}}t
		+{t}^{2} \\
		& \ \ \ \ -{\frac {a s(a) \left( a s(a)+8\,m-8 \right) }{48\,m}}{t}^{3}+{\frac {
				a s(a) \left( a s(a)+4\,m-8 \right) }{96\,m}}{t}^{4}.
	\end{align*}
	
	Recall from Lemma \ref{lem:Convexity} that the positiveness of the polynomial {$\pol(t)$ for $t>a$} is equivalent to that of the metric $g$. By writing $\tilde{\pol}$ in this way, in terms of $t$ and after dividing by $a^2,$ emphasises the scale invariance of the quantity $as(a)$. Indeed, writing $\tilde{\pol}(t)$ in terms of $y=a s(a)$ gives
	\begin{align*}
		\tilde{\pol}(t) & = \frac{y^2}{96 m} \left( (t-1)^3 (t+1)  \right) +  \\
		& \ \ \ + \frac{4y}{96 m} \left(  (t-1)^2 ((m-2)t^2 -2mt-5m+2)  \right) + \left( (t-1) (t+m-1) \right) .
	\end{align*}
	Furthermore, the smoothness conditions (\ref{eq:conditions for closing smoothly 2}) can be translated into $\tilde{\pol}(1)=0$ and $\tilde{\pol}'(1)=m>0$ and can be immediately checked from the formula above. In particular, this shows that $\tilde{\pol}$ is positive in $(0, \epsilon)$ for some positive $\epsilon$.
\end{proof}

\begin{remark}[The metrics from Theorem \ref{thm:Local existence and uniqueness} form a real $1$-parameter family of non-isometric and non-homothetic metrics]
	Suppose, for the rest of this remark that $s(a) \neq 0$ and let $\M$ be the moduli space of Bach-flat K\"ahler metrics we construct with $s(a) \neq 0$. We will keep our discussion informal in that we do not discuss topological issues. Consider the map $f:\M \to \R^2$ given by
	$$f(g)=(a^{-1},s(a)).$$ 
	Notice that $as(a)$ is invariant by scaling the metric so this map is homogeneous of with respect to the $\mathbb{R}^+$-action on $\M$ arising from scaling. Furthermore, $a$ is the volume of the unique $\bbT^2$-invariant compact divisor and $s(a)$ the scalar curvature of the ambient metric at such a divisor. This shows that the map $f$ is well defined and in fact, it descends to a map
	$$[f]: \M /\mathbb{R}^+ \to \mathbb{RP}^1 ,$$
	where $\mathbb{R}^+$ acts on $\M$ by scaling. Furthermore, the image of $[f]$  is open as for any given $(a,s)$ with $s\neq 0$, we can keep $a$ fixed and vary $s(a)$ is a neighbourhood of $s$. 
	We therefore conclude that our construction gives a continuous $1$-parameter family of non-isometric and non-homothetic Bach-flat K\"ahler metrics. 
\end{remark}

\begin{remark}[Local existence for extremal K\"ahler metrics]
	The same proof, without imposing the condition $q_3q_1=q_4q_0$ yields a local classification of $U(2)$-invariant extremal K\"ahler metrics in a neighbourhood of the zero section in $\mO(-m)$.
	
	For such metrics we need only impose the smoothness conditions \ref{eq:conditions for closing smoothly in terms of q1} which results in equations \ref{eq:smoothness conditions on q_0 and q_1} determining $q_0$ and $q_1$ from $a$, $q_3$ and $q_4$. We therefore find that these metrics are parameterised by three real parameters or two if we factor out scaling. 
\end{remark}

\subsection{Vanishing locus of the scalar curvature}

As we will be interested in applying Derdzi\'nski's theorems \ref{Thmderd1} and \ref{Thmderd2}, we want to study the zeroes of the scalar curvature of the Bach-flat K\"ahler metrics we have just constructed. We prove the following which will be useful later on. 
\begin{lemma}[Vanishing locus of the scalar curvature]\label{lem:Vanishinh of scalar curvature}
	Let $m \in \N$, $a>0$, $s(a) \in \R$ and ${g}$ be the associated Bach-flat K\"ahler metrics from Theorem \ref{thm:Local existence and uniqueness}. As a function of $r=x_1+x_2$, its scalar curvature is given by
	\begin{align*}
		\scal(g)(r) & = \frac{s(a)}{4m} \left(  a s(a) +8 (m-1) - (a s(a) +4 (m-2)) \frac{r}{a}    \right).
	\end{align*}
	From this formula it follows that:
	\begin{itemize}
		\item If $s(a)=0$, then $\scal(g)\equiv0.$
		\item If $as(a) +4 (m-2)=0$, then $\scal(g)\equiv s(a)$ is constant for all $r$.
		\item If $s(a) \neq 0$ and $a s(a) +4 (m-2) > 0$, the scalar curvature vanishes at
		$$\frac{r}{a}= \frac{a s(a) +8 (m-1) }{a s(a) +4 (m-2)} = 1 + \frac{4m}{a s(a) +4 (m-2)}$$
		if the metric is defined there.
		\item If $s(a) \neq 0$ and $a s(a) +4 (m-2)<0$, the scalar curvature never vanishes.
	\end{itemize}
\end{lemma}

\begin{proof}
	Recall that the scalar curvature can be written as $\scal(g)(r)=2q_3 + q_4 r$ and inserting the formulas for $q_3$ and $q_4$ found in equation \ref{eq:q_3 q_4 from s(a)} gives
	\begin{align*}
		\scal(g)(r) & = s(a) \frac{a s(a) +8 (m-1) }{4m} - s(a) \frac{a s(a) +4 (m-2) }{4m} \frac{r}{a} \\
		& =  \frac{s(a)}{4m} \left(  a s(a) +8 (m-1) - (a s(a) +4 (m-2)) \frac{r}{a}    \right).
	\end{align*}
	Using this formula, we find that if $s(a)=0$, then $\scal(g)\equiv0,$ while if $s(a) = -\frac{4}{a} (m-2)$, then 
	$$\scal(g)(r)= \frac{s(a)}{4m} ( a s(a) +8 (m-1) ) = s(a), $$
	and is therefore constant for all $r$.
	
	Finally, if $s(a) \neq 0$ and $a s(a) +4 (m-2) \neq 0$, then the function $\scal(g)$ vanishes at
	$$\frac{r}{a}= \frac{a s(a) +8 (m-1) }{a s(a) +4 (m-2)} = 1 + \frac{4m}{a s(a) +4 (m-2)}.$$
	In particular, if $a s(a) +4 (m-2)<0$, this is smaller than $1$ and so the scalar curvature cannot vanish.
\end{proof}

\section{Globalising the local extremal K\"ahler metrics}\label{sec:localglobal}

In this section, our goal is to extend the metrics from Proposition \ref{thm:Local existence and uniqueness} to their maximal domain of definition and understand their global properties. 

In \ref{ss:Globalizing} we shall state and prove the main globalisation/classification result. Then, in the remaining of this section, i.e. in \ref{ss:Cone angle}, \ref{ss:cscK}, \ref{ss:Incomplete} we give examples of the globalisation procedure by explicitly describing some of the resulting metrics and their properties.

\subsection{Classification result}\label{ss:Globalizing}

This subsection is dedicated to showing the following classification of $U(2)$-invariant, extremal K\"ahler metrics.

\begin{thm}\label{prop:Globalizing} 
	Let $a>0$ and let $q(r)$ be given by
	$$q(r)=q_0+q_1r+\frac{q_3}{6}r^3+\frac{q_4}{12}r^4.$$ Assume that $q$ satisfies the boundary conditions (\ref{eq:conditions for closing smoothly in terms of q1}).
	Let $g$ be the extremal K\"ahler metric determined by $q$ as in Theorem \ref{prop:scalar curvature and extremal metrics}.
	Let $\pol(r)$ be the polynomial defined by $\pol(r)=r^2-q(r)$ .
	\begin{itemize}
		\item[(i)] If $\pol(r)$ has no zero greater than $a$, then the metric $g$ from Theorem \ref{thm:Local existence and uniqueness} is defined for all $r \geq a$ and is either:
		\begin{itemize}
			\item[(i.a)] Incomplete if $\pol(r)$ has degree greater than $3$.
			\item[(i.b)] Complete if $\pol(r)$ has degree at most $3$. In this case, the metric has quartic volume growth if the degree is $2$ and exponential volume growth if the degree is $3$.
		\end{itemize}
		
		\item[(ii)] On the other hand, if instead $\pol(r)$ has a zero at some point $b>a, $ and no other zero in $[a,b],$ either:
		\begin{itemize}
			\item[(ii.a)] The zero has multiplicity one, in which case $g$ compactifies as a metric on the total space of $\mathbb{H}_m$ with a cone angle $\beta=-\frac{\pol'(b)}{mb}$ along the divisor at infinity $\lbrace r=b \rbrace$. 
			\item[(ii.b)] The zero has multiplicity greater than one, in which case the metric has a complete end as $r \to b$ and finite volume.
		\end{itemize}
	\end{itemize}
	Furthermore, any $U(2)$-invariant extremal K\"ahler metric is one of these.
\end{thm}
\begin{proof}[Proof of part (i)]
	We start by analysing the case in which $\tilde{\pol}(t)$ has no zero greater than $1$. Then, the coefficient of $t^4$ must be nonnegative. Furthermore, the corresponding metric can be written as
	$$g= \sum_{i,j=1}^2 \left( u_{ij} dx_i dx_j + u^{ij} d \theta_i d \theta_j \right),$$
	with $u$ given in terms of $q$ via equation (\ref{eq:symplectic potential}) and Remark \ref{rem:Different parametrizations}. We can consider the geodesic ray parameterised by $\ell \in (\ell_0 , +\infty)$ and given by $\ell \mapsto \curv(\ell)$ with $x_1(\curv(\ell))=\frac{\ell}{2}=x_{{2}}(\curv(\ell))$ and both $\theta_i(\curv(\ell))$ constant for $i=1,2 $. Then, we find that
	\begin{align*}
		\curv^*g & = \sum_{i,j=1}^2 ( u_{ij} ) \frac{d\ell^2}{4} \\
		& = \frac{1}{4} \left(\frac{1}{x_1} + \frac{1}{x_2} + \frac{4f(\ell)}{1-\ell f (\ell)}\right) \Big\vert_{x_1=\frac{\ell}{2}=x_2} d\ell^2 \\
		& = \frac{1}{4} \left(\frac{4}{\ell}  + \frac{1}{\ell} \frac{4q(\ell)}{\ell^2 - q (\ell)}\right)  d\ell^2 \\
		& = \frac{\ell d\ell^2}{\ell^2 - q (\ell)}   \\
		& =\frac{\ell d\ell^2}{ \pol(\ell)}.
	\end{align*}
	Now, $\pol(\ell)$ has degree $\alpha \geq 2$ and is positive for $\ell > \ell_0$ for some fixed $l_0$. Therefore, we find that 
	\begin{align}\label{incomplete}
		\mathrm{Length}(\curv) & = \int_{\ell_0}^{+\infty}  \sqrt{\frac{\ell}{ \pol(\ell)}} d \ell \sim \int_{\ell_0}^{+\infty} \ell^{-\frac{\alpha-1}{2}} d \ell ,
	\end{align}
	and this is finite if and only if $\alpha-1>2$. Thus proving that $g$ is incomplete when $\alpha>3$.
	
	Let $S_\ell$ be the subset of our manifold given by the pre-image via $T_m$ composed with the moment map of the set 
	$$\{(x_1,x_2): x_1,x_2>0, x_1+x_2 \geq a, x_1+x_2=\ell\}$$ and $\cS_\ell$ be the subset of our manifold given by the pre-image via $T_m$ composed with the moment map of the set 
	$$\{(x_1,x_2): x_1,x_2>0, a\leq x_1+x_2\leq\ell\}.$$ 
	In particular, we have $\partial \cS_\ell =S_\ell$.
	
	We now show that, if $\alpha\leq 3$ the metric is complete. Indeed, in that case we found that $\mathrm{Length}(\curv)=+\infty$ so that $\mathrm{dist}(S_\ell, S_{\ell_0}) \to +\infty$ as $\ell \to + \infty$. Any unbounded geodesic crosses any hypersurfaces $S_\ell $ for $\ell$ large enough.  Hence, when $\alpha \leq 3$ unbounded geodesics have infinite length.
	
	We finalize the argument leading to the proof of the first item by proving the assertion regarding the volume growth of these metrics. In all cases the volume form is $\frac{\omega^2}{2}= dx_1 \wedge dx_2 \wedge d\theta_1 \wedge d\theta_2$ and so the volume of the region $\cS_\ell =\lbrace r \leq \ell \rbrace$ is given by
	$$\mathrm{Vol}(\cS_{\ell}) = \int_0^{2\pi} d \theta_1 \int_0^{2\pi} d\theta_2 \int_{a \leq x_1 + x_2 \leq \ell} dx_1 \wedge dx_2 = \frac{(2\pi)^2}{2} \left( \ell^2 -a^2 \right). $$
	On the other hand, our previous computation shows that the distance $d$ to the zero section $S_a$ satisfies
	$$c_1  \sqrt{\ell}\leq d\leq c_2  \sqrt{\ell}, \ \text{if $\alpha =2$},$$
	and 
	$$c_1  \log{\ell}\leq d\leq c_2  \log{\ell}, \ \text{if $\alpha =3$,}.$$ 
	for positive constants $c_1$ and $c_2.$ This implies
	$$\cS_{ \lambda_1 R^2}\subset B_R\subset \cS_{ \lambda_2 R^2}, \ \text{if $\alpha =2$},$$
	and 
	$$\cS_{ \lambda_1 e^R}\subset B_R\subset \cS_{ \lambda_2 e^R}, \ \text{if $\alpha =3$},$$
	for constants $\lambda_1$ and $\lambda_2.$
	Inserting in the formula for the volume of the ball of radius $R$ to some fixed point in $S_a,$ $B_R$ we find that
	$$\mathrm{Vol}(B_R) \sim R^4, \ \text{if $\alpha =2$},$$
	and 
	$$\mathrm{Vol}(B_R) \sim e^{2R} , \ \text{if $\alpha =3$}.$$
\end{proof}
	
\begin{proof}[Proof of part (ii)]
	Finally, suppose that $\tilde{\pol}(t)$ attains a zero at some $t_1 > 1$ and that this has multiplicity one (the case of multiple multiplicity will be analysed later in the proof). Then, we must have $\tilde{\pol}(t_1)=0$ and $\tilde{\pol}'(t_1)\neq 0$. Equivalently, we have $b=t_1 a>a$ such that $\pol(b)=0$ and $\pol'(b)<0$. Hence, there is $p>0$ such that $\pol'(b)=-\frac{mb}{p} $ and so $g$ determines a metric with a cone angle $\beta=2\pi p^{-1}$ along the divisor $S_b$.
	
	If the zero $b$ of $\pol(r)$ has multiplicity greater than one, there is $\alpha\geq 2$ such that $\pol(r)/(b-r)^\alpha$ is smooth for $\ell \leq b$ and close to $b$. Consider the same geodesic ray $c$ as in the first part of the proof. The same computation as before gives that
	\begin{align}
		\mathrm{Length}(\curv) & \sim \int_{\ell_0}^{b} \sqrt{\frac{\ell }{ \pol(\ell)}}d \ell=+\infty ,
	\end{align}
	because 
	\begin{align}
	 \int_{\ell_0}^{b} \frac{ d \ell}{\sqrt{(b-\ell)^{\alpha}}} =+\infty,
	 \end{align}

	 and, as in the previous case, this allows us to conclude the metric is complete since any unbounded geodesic must cross $S_\ell$ for all $\ell$ large enough. The case of a triple root is similar. Furthermore, the same computation as in the proof of part (i) shows that the volume of the resulting metric is $2\pi^2 (b^2-a^2)$ and therefore finite.
\end{proof}

\subsection{Examples of cone angle metrics}\label{ss:Cone angle}
In this subsection we shall focus on case (ii.b) from Theorem \ref{prop:Globalizing}. Some of the metrics we have constructed yield metrics on $\bbH_m$ having a cone angle singularity along the divisor at infinity. Such metrics have played an important role in K\"ahler geometry as they arise in continuity methods for establishing existence of K\"ahler-Einstein or cscK metrics (see \cite{dca}). Here, they appear naturally in a family for certain choices of $(a,s(a))$ with varying cone angles. In this subsection we characterise the values of the parameters yielding cone angle metrics on $\bbH_m$ and we point out that some of our Bach-flat metrics with a cone angle on $\bbH_m$ are actually K\"ahler-Einstein. We also study the vanishing of the scalar curvature of the $U(2)$-invariant, Bach-flat metrics we discuss with an eye on applying Derdzi\'nki's results.

\begin{cor}[Metrics on $\mathbb{H}_m$ with a cone angle]\label{cor:(ii)}
	{Let $a>0$ and $s(a)\in \bbR^n.$ Let $g$} be the Bach-flat K\"ahler metric from Theorem \ref{thm:Local existence and uniqueness} associated to $(a,s(a))$. The following statements hold:
	\begin{itemize}
		\item If $m=1$ and $a s(a) \in (0,4]$, then $\scal(g)$ is positive and the metric on $\mathcal{O}(-1)$ compactifies to a metric on $\mathbb{H}_1$ with a cone angle along the divisor at infinity.  
		
		If $a s(a)=4$, then $g$ is Einstein with positive scalar curvature and cone angle along the divisor at infinity $\beta=2\pi\left(\sqrt{3}-1\right)$.
		
		\item If $m \geq 3$ and $a s(a) \in \left( - 4(m-2) , 0 \right)$, then the metric on $\mathcal{O}(-m)$ compactifies to a metric on $\mathbb{H}_m$ with a cone angle in the divisor at infinity. 

	\end{itemize} 
\end{cor}
\begin{proof}
	Again, we continue to use the notation $y= a s(a)$ and write the polynomial $\pol(r)$ in terms of $\tilde{\pol}(t)$ where $t=\frac{r}{a}$. Recall that because $\tilde{\pol}(1)=0$ and $\tilde{\pol}'(1)=m>0$ we have that $\tilde{\pol}(t)$ is positive for $t>1$ and close to $1$. {We wish to determine, wether there is $t$ greater than $1$ for which $\tilde{\pol}(t)$ vanishes}. Given that $t-1$ is positive for $t>1$, we can instead consider 
	\begin{align*}
		\frac{\tilde{\pol}(t)}{t-1} & = \frac{y^2}{96 m} \left( (t-1)^2 (t+1)  \right) +  \\
		& \ \ \ + \frac{4y}{96 m} \left(  (t-1) ((m-2)t^2 -2mt-5m+2)  \right) + \left( t+m-1 \right) ,
	\end{align*}
	which we can also write as
	\begin{align*}
		\frac{\tilde{\pol}(t)}{t-1} & =  \frac {( y+12m )  ( y +8(m-1) )}{96m}  - \frac { \left( y+12m \right)  \left( y-8 \right) }{ 96m} t \\ 
		& \ \ \ \ -{\frac {y \left( 12m+y-8 \right) }{96m}}{t}^{2}+{\frac {y \left( y+4(m-2) \right) }{96m}}{t}^{3}.
	\end{align*}
	Therefore, if $y \left( y+4(m-2) \right) <0$, i.e.
	$$y \in \left(\min \lbrace 0, -4(m-2) \rbrace , \max \lbrace 0, -4(m-2) \rbrace\right),$$
	$\tilde{\pol}(t)$ crosses zero for some $t_{m,y}>1$ and so the metric is only defined on {the pre-image via $T_m$ composed with the moment map  of $\{(x_1,x_2):x_1,x_2>0, a<x_1+x_2<t_{m,y}\}$}. However, in order to apply Theorem \ref{prop:Globalizing} to prove the claimed statement in (ii.a) we must prove that $t_{m,y}$ has multiplicity one as a root of $\tilde{\pol}(t)$, or equivalently as a root of $\frac{\tilde{\pol}(t)}{t-1}$. We claim that this is a consequence of the fact that
	\begin{align*}
		\frac{d}{dt} \left( \frac{\tilde{\pol}(t)}{t-1} \right) \Big\vert_{t=1} = 1-\frac{y}{4} >0 , \ \text{for $y <4 \leq \max \lbrace 0, -4(m-1) \rbrace$.}
	\end{align*}
	Indeed, $\frac{\tilde{\pol}(t)}{t-1}$ is a degree $3$ polynomial and so has at most two critical points. One must be smaller than $1$ as in this regime $\frac{\tilde{\pol}(t)}{t-1}$ is decreasing for very small $t$ and increasing at $t=1$ by the previous computation. There is another critical point in the interval $[1,t_{m,y}]$ as $\frac{\tilde{\pol}(t)}{t-1}$ is positive and increasing at $t=1$ while it is decreasing for $t=t_{m,y}-\epsilon$ and $\epsilon \ll 1$. As a consequence, $t_{m,y}$ itself cannot be a third critical point.
	
	As for the statement regarding the scalar curvature, we know from Lemma \ref{lem:Vanishinh of scalar curvature} that if $y+4(m-2)<0$ the scalar curvature never vanishes. Notice however that this condition together with $y \in \left(\min \lbrace 0, -4(m-2) \rbrace , \max \lbrace 0, -4(m-2) \rbrace\right)$ implies $m=1$ and $y \in (0,4)$.
	
	We are still missing the case $m=1$ and $a s(a)=4$. In this case we find that 
	$$\tilde{\pol}(t)=- \frac{t-1}{3} \left( t^2-2t-2 \right),$$
	which for $t \in [1,1+\sqrt{3}]$ is nonnegative thus defining a metric on {the pre-image via $T_m$ composed with the moment map  of $\{(x_1,x_2):x_1,x_2>0, a<x_1+x_2<b\}$}. To compute the cone angle at $r=b$ we use the formula in (ii.b) of Theorem \ref{prop:Globalizing} which gives  $p=-\frac{mb}{p'(b)}=\frac{1}{\sqrt{3}-1}$. 
	Also, in this case we find that $q_4=0$ so $a\scal(g)(r)=2aq_3 = as(a) =4 $ and the scalar curvature is constant. From Derdzi\'nski's theorem \ref{Thmderd2}, $(\scal(g))^{-2} g$ is Einstein which, as $\scal(g)$ is constant, implies that $g$ itself is Einstein.
\end{proof}

\subsection{Examples of complete cscK metrics}\label{ss:cscK}
In this subsection we shall turn our attention to case (i.b) from Theorem \ref{prop:Globalizing}. This specific parameter choice for $(a,s(a))$ yields constant scalar curvature complete metrics on the total space of $\mO(-m).$ Toric scalar-flat metrics on $\mathcal{O}(-m)$ are completely understood (see \cite{s}) and classified. The metrics we are discussing here are special among those since they have an extra $U(2)$ symmetry. However, little is known about non scalar-flat cscK metrics in the same setting. We derive some information about the asymptotic behaviour near infinity. 
\begin{cor}[Complete metrics on $\mathcal{O}(-m)$ with constant scalar curvature]\label{cor:scalar flat}
	{Let $a>0$ and $s(a)\in \bbR^n.$ Let $g$}  be the Bach-flat K\"ahler metric from Theorem \ref{thm:Local existence and uniqueness} associated to $(a,s(a))$. Then, the following hold:
	\begin{itemize}
		\item If $m \geq 1$ and $a s(a)=0,$ the metric $g$ on $\mathcal{O}(-m)$ is complete and scalar-flat. 
		
		\item If $m \geq 3$ and $a s(a) =-4(m-2),$ then $g$ is complete and K\"ahler--Einstein. 
	\end{itemize}
	In particular, these metrics have quartic volume growth in the scalar-flat case and exponential volume growth in the negative Einstein case. 
\end{cor}
\begin{proof}
	Let us study the case when $y \left( y+4(m-2) \right) =0$. Notice that in such a situation $g$ has constant scalar curvature by Lemma \ref{lem:Vanishinh of scalar curvature}.
	
	We start with the case in which $y=0$. Then,
	\begin{align*}
		\frac{\tilde{\pol}(t)}{t-1} & = \left( (m-1) +  t \right),
	\end{align*}
	which is therefore positive for all $t>1$ and so the metric is defined on the whole $\mathcal{O}(-m)$. In this case, the polynomial $\pol$ has degree $2$ and by (i.b) of Theorem \ref{prop:Globalizing} we find that the metric is complete with quartic volume growth.
	
	As for the case in which $y+4(m-2)=0$ we find that
	\begin{align*}
		\frac{\tilde{\pol}(t)}{t-1} & = \frac{1}{3} \left( m+1 + (m+1) t  + (m-2 ) {t}^{2} \right).
	\end{align*}
	The roots of this polynomial are $-1$ for $m=2$, and
	$$\frac{-(m+1) \pm \sqrt{-3m^2+6m+9}}{2(m-2)},$$
	for all other $m>2$. If real, these are always smaller than $1$ and so ${\pol}(t)$ always remains positive for $t >1$. 
	
	In this last case, the polynomial ${\pol}$ has degree $3$. Therefore, by item (i.b) of Theorem \ref{prop:Globalizing} we find that the resulting metric is complete with exponential volume growth.
	
	Finally, if the scalar curvature $\scal(g)$, which is constant, is nonzero, then by Derdzi\'nski's theorem \ref{Thmderd2}, $(\scal(g))^{-2}g$ is Einstein. As $\scal(g)$ is constant, $g$ is also Einstein.
\end{proof}

\begin{rem}
	By our uniqueness result, the scalar-flat K\"ahler metric corresponding to $a s(a)=0$ and $m=2$ must coincide with the Eguchi-Hanson metric and so it is actually Ricci-flat.
\end{rem}

\subsection{Examples of incomplete metrics}\label{ss:Incomplete}
In this subsection we shall focus on case (i.a) of Theorem \ref{prop:Globalizing}. Although the metrics we obtain here are incomplete we can apply Derdzi\'nski's Theorem \ref{Thmderd2} to them on the region where their scalar curvature remains non-vanishing to get complete Einstein metrics just as we will apply the theorem in the case of $U(2)$-invariant, Bach-flat metrics with a cone angle on $\bbH_m.$
\begin{cor}[Incomplete metrics on $\mathcal{O}(-m)$]\label{cor:(iii)}
	{Let $m\in \mathbb{N},$ $a>0$ and $s(a)\in \bbR^n.$ Let $g$}  be the Bach-flat K\"ahler metrics from Theorem \ref{thm:Local existence and uniqueness} associated to $(a,s(a)).$ Then, there is $y_m \gg 1$ such that if $a s(a) >y_m$, the metric on $\mathcal{O}(-m)$ is incomplete. Furthermore, in this situation either:
	\begin{itemize}
		\item the scalar curvature is everywhere negative if $s(a)<0$;
		\item or, if $s(a)>0$, the scalar curvature vanishes along the hypersurface whose image via $T_m$ composed with the moment map is characterised by 
		$$\frac{r}{a}=  1 + \frac{4m}{a s(a) +4 (m-2)}.$$
	\end{itemize}
\end{cor}

\begin{proof}
	Again we start by writing 
	\begin{align*}
		\frac{\tilde{\pol}(t)}{t-1} & = \frac{y^2}{96 m} \left( (t-1)^2 (t+1)  \right) +  \\
		& \ \ \ + \frac{4y}{96 m} \left(  (t-1) ((m-2)t^2 -2mt-5m+2)  \right) + \left( t+m-1 \right).
	\end{align*}
	Set $t=1+\frac{s}{y}$. This gives
	\begin{align*}
		\frac{\tilde{\pol} \left(1+\frac{s}{y} \right)}{\frac{s}{y}} & = {\frac { \left( 4\,m+y-8 \right) {s}^{3}+ \left( 2\,{y}^{2}-16\,y	\right) {s}^{2}-24\,my \left( y-4 \right) s+96\,{y}^{2}{m}^{2}}{96\,m {y}^{2}}}.
	\end{align*}
	This cubic polinomial has a root at $s=0$ and discriminant given by the formula
	\begin{align*}
		& -768\,{m}^{2}{y}^{3} \left( 6\,m-12+y \right)  \left( 12\,m+y \right) \times \\
		& \ \ \ \  \times \left( 72\,{m}^{2}y+18\,m{y}^{2}+{y}^{3}-144\,my-12\,{y}^{2}+256
		\right) ,
	\end{align*}
	whose leading term is $-768m^2 y^8$ and thus becomes negative for $y \gg 1$. In such circumstances the only zero of $\frac{\tilde{\pol}(t)}{t-1}$ occurs at $t=1$.
	
	Finally, in such a situation, and by possibly increasing $y_m$, we clearly have $y+4(m-2)>0$ as $y \geq y_m$. Thus, according to Lemma \ref{lem:Vanishinh of scalar curvature} the scalar curvature must then vanish along a hypersurface.
\end{proof}

\subsection{Examples of finite volume metrics}\label{ss:finite volume}

	Let $t_1=b/a$, then $\tilde{\pol}(t)$ has a zero at $t=t_1$. If this zero has multiplicity greater than one, $\frac{\tilde{\pol}(t)}{t-1}$ must have a zero at $t_1$ with multiplicity greater than one and so its discriminant must vanish. We find that
	\begin{align*}
		\frac{\tilde{\pol}(t)}{t-1} & =  \frac {( y+12m )  ( y +8(m-1) )}{96m}  - \frac { \left( y+12m \right)  \left( y-8 \right) }{ 96m} t \\ 
		& \ \ \ \ -{\frac {y \left( 12m+y-8 \right) }{96m}}{t}^{2}+{\frac {y \left( y+4(m-2) \right) }{96m}}{t}^{3}.
	\end{align*}
	and up to a positive proportionality constant, as we have seen the discriminant of this polynomial is given by a constant times
	$$-y^3 ( y+12m )  ( 6m-12+y )  ( 72{m}^{2}y+18m{y}^{2}+{y}^{
		3}-144ym-12{y}^{2}+256 )  .
	$$
	Hence, we find that $\pol$ has a double root if either $y=-12m$, $y=-6(m-2)$, or ${y}^{3}+ 6( 3m-2 ) {y}^{2}+ 72 m ( m-2 ) y+256=0$. When $y=-12 m$ the resulting double root is at $t=\frac{1}{m+1}<1$ and so we can disregard this case. As for when $y=-6(m-2)$ we find that resulting double root is positive (and greater than $1$) only in the case $m=1$ in which case it occurs at $t_1=3$. We shall encounter the corresponding metric again later in the proof of Proposition \ref{prop: Einstein cone angle} and show that this complete, Bach-flat, K\"ahler metric is conformal to a Ricci-flat ALF space namely the smooth Taub-bolt space. This has been studied in detail in  \cite{BG}.

\section{Bach-flat K\"ahler metrics on $\bbH_m$ with a cone angle}\label{sec:hm}

We are now going to look more in-depth into the metrics arising from Corollary \ref{cor:(ii)}. Our goal is to use these as a starting point in Derdzi\'nski's construction. In most cases, as we shall see, the scalar curvature of the extremal metric vanishes along a hypersurface in $\bbH_m$ which separates it into two connected components. Hence, we can only apply Derdzi\'nski's theorem to each of these connected components separately to obtain  what are called Poincar\'e-Einstein metrics. 

The main difference between this section and Corollary  \ref{cor:(ii)} is that rather than using $(a,s(a))$ as parameters, we are going to use the weight $p$ along the divisor at infinity on $\bbH_m$ and $x=b/a$ the quotient of the volume of the section at infinity by the volume of the zero section as parameters. We do in order to have a more geometric interpretation of our construction. 

\subsection{Existence} 

	\begin{prop}\label{prop:metrics with cone-angles}
		Let $m \in \N$, $b>a$ and write $x=\frac{b}{a}$. Then, {for any $p$} there is {an explicit} $U(2)$-invariant, extremal K\"ahler metric $g$ in $\mathbb{H}_m$ with a cone angle $p$ {along} the divisor at infinity. 
		
		This metric is Bach-flat, if and only if
		\begin{equation}\label{eq:quadratic equation in p}
			Ap^2+Bp+C=0,
		\end{equation}
		for
		\begin{align*}
			A & = \left( 3\,{x}^{2}m+2\,{x}^{3}+3\,xm-2 \right)  \left( xm+m-2\,x+2
			\right) \\
			B & = - m \left(  \left( m-2 \right) {x}^{4}+ \left( 2\,m+4 \right) {x}^{3
			}+6\,{x}^{2}m+ \left( 2\,m-4 \right) x+m+2 \right)  \\
			C & = 3\,{m}^{2} x \left( x+1 \right) ^{2} .
		\end{align*}
	\end{prop}
	
	{\begin{rem}
	The metrics in the above proposition correspond to the globalisation of those arising from Theorem \ref{thm:Local existence and uniqueness} with
		\begin{equation}\label{eq:matching cone angle and curvature}
			a s(a)=	{\frac { 12 \left(\left( 2\,p -m \right) {x}^{2}+ \left(  \left( 3\,m-2
					\right) p-m \right) x+mp\right)}{ \left( {x}^{2}+4\,x+1 \right)  \left( x-1
					\right) p}}.
		\end{equation} 

	\end{rem}}
	\begin{proof}
	Recall that we need to impose the condition that
	$$
	h''(r)-\frac{1}{m}\left(\frac{1}{r-a}+\frac{p}{b-r}\right)
	$$
	is smooth on $[a,b]$. In terms of $q(r)$, we have $h''=\frac{f}{1-rf}= \frac{1}{r} \frac{q}{r^2-q}$ and the smoothness condition amounts to
	\begin{equation}\label{eq:conditions for closing smoothly}
		\begin{aligned} 
			q(a) = a^2 & , \quad q(b)=b^2 \\ 
			q'(a)= (2-m) a & , \quad q'(b)= \left( 2+ \frac{m}{p} \right) b . 
		\end{aligned}
	\end{equation}
	This is a linear system on the coefficients of the quartic polynomial $q(r)$. Furthermore, since  $q_2 =0,$ such a system has a unique solution for $a \neq b$. Using the notation $x=\frac{b}{a}$, the solution is
	\begin{equation}\label{eq:q's for cone angle}
	\begin{aligned}
		q_0 & =\frac {{a}^{2} \left( p \left( m-1 \right) {x}^{2}+ \left(  \left( 2
			m+2 \right) p-2m \right) x-m-p \right) {x}^{2}}{p \left( x-1
			\right) ^{2} \left( {x}^{2}+4x+1 \right) } ,\\
		q_1 & =-{\frac {a ( p ( m-2 ) {x}^{3}+ (  ( 2m+2
				) p-3m ) {x}^{2}+ (  ( 3m+2 ) p-2m
				) x-m-2p ) x}{p ( x-1 ) ^{2} \left( {x}^{2}+
				4x+1 \right) }}, \\
		q_3 & = \frac { 6( 2p-m ) {x}^{3}+ (  ( 18m-12
			) p-12m ) {x}^{2}+ ( 12 ( m-1 ) p-18
			m ) x+6p ( m+2 ) }{ap ( x-1 ) ^{2}
			( {x}^{2}+4x+1 ) },\\
		q_4 & = \frac { 24\left( m-p \right) {x}^{2}+ \left(  \left( 1-m
			\right) p+m \right) 48x-24 \left( m+1 \right) p}{{a}^{2}p \left( 
			x-1 \right) ^{2} \left( {x}^{2}+4x+1 \right) }.
	\end{aligned}
	\end{equation}
	{The first thing to do is show that for such choices of $q_0, q_1, q_3$ and $q_4$, the resulting polynomial $\pol$ given in terms of $q$ via Remark \ref{rem:Different parametrizations} is positive in $[a,b]$ or said differently that $b$ is the smallest zero of $\pol$ greater than $a.$ This a technical point and we leave it to Appendix \ref{appendix:positive_p}}
	
	{Replacing the above formulas for the $q$'s in the equation $q_3 q_1  = q_4 q_0$, gives} a quadratic equation for $p$ which we write in standard form as $Ap^2+Bp+C=0$,
	with $A$, $B$, and $C$ given by the formulas in the statement. 
	
	Finally, in order to match this metric with the metrics arising from Theorem \ref{thm:Local existence and uniqueness} we must only compute the scalar curvature of these metrics at $r=a$. Recall from Proposition \ref{prop:scalar curvature and extremal metrics} that $s(r)=2q_3 + q_4 r$. Then, evaluating at $r=a$ and inserting the values for $q_3$, $q_4$ found we obtain the formula in the statement.
	\end{proof}

	\begin{remark}
	The quadratic equation (\ref{eq:quadratic equation in p}) has a solution if and only if the discriminant $B^2-4AC$ is non-negative. The discriminant is given by $36m^{2}  x^{2}
	\left( {x}^{2}+4x+1 \right)^{2}$ times 
	$$   \left( m-2 \right)^{2} x^{4} - 4m \left( m -2 \right) x^{3} - 2 \left( 3m^{2}+4 \right) x^{2} - 4m \left( m+2 \right) x+ \left( m+2 \right) ^{2} .$$
	Later, we shall investigate the sign of this quantity separately for $m=1$, $2$ and $m\geq 3$.
	\end{remark}

\subsection{Vanishing locus of the scalar curvature} To apply Derdzi\'nski's theorem to the metrics above we must understand the vanishing locus of the scalar curvature. 
\begin{prop}\label{proposition_admissible}
	Let $m \in \N$, $b>a>0$, $x=\frac{b}{a}$ and $g$ the extremal K\"ahler metric from Proposition \ref{prop:metrics with cone-angles}. Then, there are three rational functions of $m$ and $x$ 
	\begin{equation}
		\begin{aligned}\nonumber
			r_1 & =\frac{mx(x+1)}{2x(x-1)+m(3x+1)},\\  \nonumber
			r_2 & =\frac{m x (x+3)}{2-2x+m(x+1)}, \\ \nonumber
			d & =\frac{mx(x + 2)}{(x-1)^2 + m(2x+1)},
		\end{aligned}
	\end{equation}
	having the following significance:
	\begin{itemize}
		\item If $m>1$ or $m=1$ and $x<3$. Then, the cone angle extremal K\"ahler metric $g$ has nowhere vanishing scalar curvature if and only if $p\in(r_1,r_2)$.
		\item If $m=1$ and $x=3$. Then, the cone angle extremal K\"ahler metric $g$ has nowhere vanishing scalar curvature if and only if $p > \frac{6}{11}$.
		\item If $m=1$ and $x>3$. Then, the cone angle extremal K\"ahler metric $g$ has nowhere vanishing scalar curvature if and only if $p>r_1$.
	\end{itemize}
\end{prop}
\begin{proof}
	Recall that the scalar curvature of the metric is given by ${\scal(g)}(r)=2q_3+q_4 r$. Notice that if $q_3 \neq 0$ and $q_4 =0$ the scalar curvature is constant and nowhere vanishing. Using the expression for $q_4$ in equation (\ref{eq:q's for cone angle}) we can rewrite the condition that $q_4=0$ as 
	$$p=d,$$ 
	{which corresponds to the case when the scalar curvature of the constructed $g$ is constant. When $\scal(g)$ is not constant it vanishes in $\bbH_m$ when $-\frac{2q_3}{q_4}$ satisfies $ a\leq-\frac{2q_3}{q_4}\leq b$} which we rewrite
	$$1\leq -\frac{2q_3}{aq_4}\leq x.$$
	We shall make this condition explicit in terms of $x,m$ and $p.$ From equation (\ref{eq:q's for cone angle}) we have
	$$
	-\frac{2q_3}{aq_4}=\frac{m x (3 + 2 x + x^2)- p \left(2 (-1 + x)^2 (1 + x) + m (1 + 2 x + 3 x^2)\right)}{2(x (x + 2) -  p (x^2 + 2 x (m - 1) + m+1))},
	$$
	and the inequality above becomes
	$$
	1\leq \frac{m x (3 + 2 x + x^2)-p \left(2 (-1 + x)^2 (1 + x) + m (1 + 2 x + 3 x^2)\right)}{2(mx (x + 2) -  p (x^2 + 2 x (m - 1) + m+1))} \leq x,
	$$
	that is
	$$
	\begin{cases}
		\frac{(x-1)\left(mx(x+1)-p(2x(x-1)+m(3x+1))\right)}{2\left(mx(x + 2) - p ((x-1)^2+m(2x+1)))\right)}\geq 0\\\
		\frac{(x-1)\left(m x (x+3)-p(2-2x+m(x+1))\right)}{2\left(mx(x + 2) - p ((x-1)^2+m(2x+1)))\right)}\geq 0,
	\end{cases}
	$$
	i.e.
	$$
	\begin{cases}
		\frac{mx(x+1)-p(2x(x-1)+m(3x+1))}{mx(x + 2) - p ((x-1)^2+m(2x+1)))}\geq 0\\\
		\frac{m x (x+3)-p(2-2x+m(x+1))}{mx(x + 2) - p ((x-1)^2+m(2x+1)))}\geq 0.
	\end{cases}
	$$
	
	For fixed $m$ and $x$ we will look at these inequalities as inequalities in $p.$ We want to understand the sign of 
	\begin{IEEEeqnarray}{lCr} \nonumber
		N_1=mx(x+1)-p(2x(x-1)+m(3x+1)),\\  \nonumber
		N_2=m x (x+3)-p(2-2x+m(x+1)), \\ \nonumber
		\mD=mx(x + 2) - p ((x-1)^2 + m(2x+1)),
	\end{IEEEeqnarray}
	whose roots are respectively
	\begin{IEEEeqnarray}{lCr} \nonumber
		r_1=\frac{mx(x+1)}{2x(x-1)+m(3x+1)},\\  \nonumber
		r_2=\frac{m x (x+3)}{2-2x+m(x+1)}, \\ \nonumber
		d=\frac{mx(x + 2)}{(x-1)^2 + m(2x+1)}.
	\end{IEEEeqnarray}
	
	When $m>1$, we have $r_1<d$ as
	$$
	\begin{aligned}
		& x+1<x+2\\
		& (x-1)^2 + m(2x+1)<2x(x-1)+m(3x+1),
	\end{aligned}
	$$
	and also $d<r_2$ since
	$$
	\begin{aligned}
		&x+2<x+3\\
		&0<2-2x+m(x+1)<(x-1)^2 + m(2x+1).
	\end{aligned}
	$$
	As for the $m=1$ case, the same holds as long as $x<3$, otherwise $r_2<0.$ 
	We can draw a sign table for $N_1/\mD$ and $N_2/\mD.$ For the case $m>1$ or $m=1$ and $x<3$ we have
	\begin{center}
		\begin{tabular}{|| c c c c c c c c c c c} 
			\hline
			$p$ &0 & & $r_1$& &$d$& &$r_2$& &  \\ 
			\hline
			$\mD$& &+ & & + &0 &- & &-&  \\ 
			\hline
			$N_1$& &+ &0 &-& &- & &-&   \\
			\hline
			$N_2$& &+ & & + & &+ & 0&-&  \\
			\hline
			$N_1/\mD$& &+&0&- & $\times$&+& &+ \\
			\hline
			$N_2/\mD$& &+& &+&$\times$&-&0&+ \\
			\hline\hline
		\end{tabular}
	\end{center}
	
	The last two rows are simultaneously nonnegative if and only if $p\leq r_1,$ or $p\geq r_2.$ Therefore for such $m$ and $x$ the scalar curvature $\scal(g)$ vanishes if and only if $p\leq r_1,$ or $p\geq r_2,$ as claimed. 
	
	The case when $m=1$ and $x>3$ can be summed up in the following sign diagram. 
	\begin{center}
		\begin{tabular}{|| c c c c c c c c c } 
			\hline
			$p$ &0 & & $r_1$& &$d$& &  \\ 
			\hline
			$\mD$& &+ & & + &0 &- &   \\ 
			\hline
			$N_1$& &+ &0 &-& &- &   \\
			\hline
			$N_2$& &+ & & + & & + &   \\
			\hline
			$N_1/\mD$& &+&0&- & $\times$&+  \\
			\hline
			$N_2/\mD$& &+& &+&$\times$&-& \\
			\hline\hline
		\end{tabular}
	\end{center}
	The last two rows are simultaneously nonnegative when $p \leq  r_1$. Therefore for fixed $m=1$ and $x>3,$ $\scal(g)$ vanishes in $[a,b]$ if and only if $p \leq  r_1$. 
	
	As for the case when $m=1$ and $x=3,$ we have
	$$\scal(g)(r) = \frac{9(11p-9)}{11 a p}  + \frac{3(15-11p) }{11 a p} \frac{r}{a}, $$
	which is constant and nonzero when $p=\frac{15}{11}$. Therefore, we may assume that $p \neq \frac{15}{11},$in which case we find that
	\begin{align*}
		\frac{N_1}{\mathcal{D}} & =  \frac{2(11p-6)}{11p-15},\\
		\frac{N_2}{\mathcal{D}} & = - \frac{18}{11p-15} ,
	\end{align*}
	The first of these is nonnegative if and only if $p \in \left( -\infty , \frac{6}{11} \right]  \cup \left( \frac{15}{11} , +\infty \right)$ while the second is nonnegative if and only if $p < \frac{15}{11}$. Therefore, they are simultaneously nonnegative for $p \leq \frac{6}{11}$.
\end{proof}

\begin{remark}
	Consider the case $m > 1$ and large $x \gg 1$. We find that
	\begin{align*}
		r_1 & = \frac{m}{2} + m \left(1-\frac{3m}{4}\right) \frac{1}{x} +{O\left(\frac{1}{x^2}\right)} \\
		r_2 & = \frac{m}{m-2} x + {O\left(\frac{1}{x^2}\right)},
	\end{align*}
	where ${O\left(\frac{1}{x^2}\right)}$ denotes a function whose absolute value is bounded by a constant times $\frac{1}{x^2}$. 
\end{remark}

\subsection{The case of $\mathbb{H}_1$}

\begin{prop}\label{cor:m=1}
	Let $a,b$ be such that $0<a<b$ and set $x = \frac{b}{a} \geq 1.$ Let $g$ be a $U(2)$-invariant Bach-flat K\"ahler metric on $\mathbb{H}_1$ with cone angle $\beta=2\pi p^{-1}$ along {$S_b$}. Then, there exist $x_0 \approx 2.61$ and functions $$p_-:[x_0,3]\rightarrow \bbR, \quad p_+:[x_0,+\infty[\rightarrow \bbR$$ such that either:
		\begin{itemize}
			\item $x=x_0$ and $p=p_{+}(x_0)=p_{-}(x_0)\approx 2.062$ for $x_0 \approx 2.61.$
			\item $x \in (x_0, 3)$ and $p=p_+(x)$ or $p=p_-(x)$. 
			\item $x=3$ and $p=\frac{12}{11}$.
			\item $x>3$ and $p=p_+(x)$. This satisfies $\lim_{x \to + \infty}p_+(x)=\frac{1}{2}$.
		\end{itemize}
		In addition, the scalar curvature $\scal(g)$ of these metrics is always positive.
		
		Now, for $x \geq x_0$ and $p_\pm(x)$ as above, let $g_{x,\pm}$ be the corresponding metrics. Then, as $x \to 3^-$, the metrics $g_{x,-}$ converge smoothly on compact subsets of $\mathbb{H}_1 \backslash S_b$ to the complete, Bach-flat, K\"ahler metric of finite volume with $a s(a)=6$ discussed in subsection \ref{ss:finite volume}.
		
		On the other hand, the limit as $x \to +\infty$ of the metrics associated with $p_+(x)$ converge smoothly (on compact sets away from $S_b$) to the complete scalar-flat ALE K\"ahler metric in $\mathcal{O}(-1)$ from Corollary \ref{cor:scalar flat}.
	\end{prop}

In fact we have explicit formulas for $p_{\pm}$ namely
\begin{equation}\label{eq:solution to quadratic eq for p}
	p_{\pm}(x)=\frac {  x^{4}-6x^{3}-6x^{2}+2x-3 \pm ( x^{2}+4x+1 ) \sqrt {x^{4}+4x^{3}-14x^{2}-12x+9} }{4x^{4}-6x^{3}-12x^{2}-22x+12}.
\end{equation}

\begin{proof}
This is a consequence of the behaviour of the quadratic equation (\ref{eq:quadratic equation in p}) when $m=1.$ We have $Ap^2+Bp+C=0$, with
		\begin{align*}
			A & = \left( 3\,{x}^{2}+2\,{x}^{3}+3\,x-2 \right)  \left( -x+3 \right) \\
			B & = \left( {x}^{4}-6 {x}^{3
			}-6\,{x}^{2}+2x-3 \right)  \\
			C & = 3\,x \left( x+1 \right) ^{2} .
		\end{align*}
The discriminant of $Ap^2+Bp+C$ is $x^{4}+4x^{3}-14x^{2}-12x+9$ and $x_0$ is the only positive root of this quartic polynomial. This discriminant is negative for $x<x_0$ and the equation only has solutions for $x\geq x_0$ and the solutions are $p_{\pm}$ as in the formulas above. Notice that the formula for $p_-(x)$ is only valid in case $A \neq 0$ which for $x>1$ is equivalent to $x\neq 3$. For $x=3$ the denominator vanishes and we have instead the solution $p=\frac{12}{11}$ obtained either from $p_+(x)$ or the equation $Bp+C=0$.
		
		Next, we find from direct inspection that while for $x\in (x_0,3)$ both $p_\pm(x)$ are positive while for $x>3$ only $p_+(x)$ is positive.

		We prove the claimed assertion on the scalar curvature using Proposition \ref{proposition_admissible}. First consider $x<3$, then according to Proposition \ref{proposition_admissible}, $\scal(g)$ does not vanish if $p \in (r_1, r_2)$ where 
		$$r_1 = \frac{x(x+1)}{2x^2+x+1} , \ \ r_2 = \frac{x(x+3)}{3-x} . $$
		One can check directly that if $x_0<x<3$, then
		\[
		\scalebox{.9}{$\frac{x(x+1)}{2x^2+x+1}<\frac{-3 + 2 x - 6 x^2 - 6 x^3 + x^4 \pm (1 + 4 x + x^2)\sqrt{9 - 12 x - 14 x^2 + 4 x^3 + x^4}}{2 (-3 + x) (-2 +3 x + 3 x^2 + 2 x^3)}<\frac{x(x+3)}{3-x},$}
		\]
		thus proving that the scalar curvature does not vanish for $x<3$. {Since $s(a)>0$}, it is everywhere positive. On the other hand, for $x>3$, Proposition \ref{proposition_admissible} ensures that the scalar curvature does not vanish if and only if $p>r_1$ which holds for $x >3$. From Proposition \ref{proposition_admissible}, for $x=3$ we need $p$ to be greater than $\frac{6}{11}$ which is the case because $p=\frac{12}{11}.$\\
		
		We turn to the second part of the statement concerning the limits of the above metrics.
		
		In order to study the limit $x \to 3^-$ of the metrics associated with $p_-(x)$ we need to look at the asymptotic behaviour of $p_-(x)$. 
		For $x<3$ near $3$
		$$p_-(x) = -\frac{3}{2} \frac{1}{(x-3)}+ O(1).$$
		where $O(1)$ denotes a bounded function in a neighbourhood of $3.$ Inserting this into equation (\ref{eq:matching cone angle and curvature}) (with $m=1$ and $p=p_-(x)$) we find that the resulting metrics have $as(a)$ given by
		\begin{align*}
			a s(a) &  = 6 + O\left(x-3\right) ,
		\end{align*}
		where $O\left(x-3\right)$ denotes a function which is bounded in absolute value by a constant times $x-3$.
		
		On the other hand, in the limit $x \to +\infty$, we have $p_+(x) = \frac{1}{2} +O(x^{-1}).$  Here again $O\left(\frac{1}{x}\right)$ denotes a function which is bounded in absolute value by a constant times $\frac{1}{x}.$ Inserting into equation (\ref{eq:matching cone angle and curvature}) we obtain that $a s(a)$ converges to $0$. Hence, the resulting metrics converge to the complete scalar-flat ALE metric on $\mathcal{O}(-1)$ from Corollary \ref{cor:scalar flat}.
	\end{proof}

	\begin{remark}[Limit of the metric associated with $p_-(x)$ as $x \to 3^-$]
	We have seen that in the limit $x \to 3^-$ the metrics associated with $p_-(x)$ converge to a complete, Bach-flat, K\"ahler metric of finite volume with $a s(a)=6$ discussed in subsection \ref{ss:finite volume}. We will encounter this metric again in Proposition \ref{prop: Einstein cone angle} where we show it is conformal to the smooth Taub-bolt metric.
	\end{remark}

\subsection{The case of $\bbH_2$}
\begin{prop}
		There is no $U(2)$-invariant Bach-flat K\"ahler metric on $\mathbb{H}_2$ with a cone angle along the divisor at infinity.
	\end{prop}
\begin{proof}
	{Recall that the discriminant of $Ap^2+Bp+C$ is given by
	{\tiny	$$  36m^{2}  x^{2}
	\left( {x}^{2}+4x+1 \right)^{2}\left[ \left( m-2 \right)^{2} x^{4} - 4m \left( m -2 \right) x^{3} - 2 \left( 3m^{2}+4 \right) x^{2} - 4m \left( m+2 \right) x+ \left( m+2 \right) ^{2} \right].$$}
	In the case $m=2$ this becomes a constant multiple of
		$$ x^{2}
	\left( {x}^{2}+4x+1 \right)^{2}\left(- 32 x^{2} - 32 x+ 16 \right).$$}
		This is non-negative for $x \in [-\frac{1}{2}-\frac{\sqrt{3}}{2}, -\frac{1}{2}+\frac{\sqrt{3}}{2}]$ and negative otherwise. However, in this interval $x<1$ and so we conclude that there is no solution to the quadratic equation (\ref{eq:quadratic equation in p}) for $x>1.$	This proves the proposition. \end{proof}

\subsection{The case of $\mathbb{H}_m$ for $m \geq 3$}
\begin{prop}\label{cor:m>2}
		Let $0<a<b$ and $m \geq 3$. Set $x=b/a.$ Let $\bbH_m$ be endowed with a K\"ahler form such that the image via $T_m$ composed with the moment map  of  $\bbH_m$ is
				$$ \{(x_1,x_2): x_1,x_2\geq 0, a\leq x_1+x_2\leq b\}.$$
		Then, {for large enough x,} there are two $U(2)$-invariant, Bach-flat, K\"ahler metrics $g_{x,\pm}$ on $\mathbb{H}_m$ with a cone angle $2\pi/p_{\pm}(x,m)$ along $S_b$. In addition, all such metrics have their scalar curvature vanishing along a $3$-dimensional hypersurface.
		
		We have 
		$$\lim_{x \to + \infty} p_+(x,m)=\frac{m}{2}, \ \ \lim_{x \to + \infty} p_-(x,m)=0.$$
		Also, in this $x \to +\infty$ limit, the metrics $g_{x,\pm}$ converge uniformly with all derivatives on compact sets to the complete, Bach-flat K\"ahler metrics $g_\pm$ on $\mathbb{H}_m \backslash {S_b}$. 
		
		The metric $g_+$ is the scalar-flat ALE K\"ahler metric in the first bullet of Corollary \ref{cor:scalar flat}. The metric $g_-$ is the negative K\"ahler--Einstein metric of exponential volume growth from the second bullet of Corollary \ref{cor:scalar flat}.
	\end{prop}
	\begin{proof}
	 Consider equation (\ref{eq:quadratic equation in p}) and the expressions that follow for $A,B,C$ in Proposition \ref{prop:metrics with cone-angles}. As $x$ tends to infinity we have 
\begin{align*}
		A &=2x^4  ( m-2 ) + O\left(x^3\right) \\
		B & = -m(m-2)x^4 +O\left(x^3\right)  \\
		C & = 3 m^{2}x^{3} + O\left(x^2\right) ,
	\end{align*}
where again $O(f(x))$ means a function which in absolute value is bounded by a constant times $f(x)$.  We see that for $x$ large enough, $B^2-4AC$ is positive. Furthermore, for sufficiently large $x$ we have $B<0$ and $AC>0$. Hence, both $-B \pm \sqrt{B^2-4AC}$ are positive and the two roots $p_\pm(x,m)$ of $Ap^2+Bp+C$ are positive. 
To find the asymptotic behaviour of $p_{\pm}(x,m)$ we simply apply the resolvent formula to equation  (\ref{eq:quadratic equation in p}) using the explicit formulas for $A$, $B$ and $C$ in Proposition \ref{prop:metrics with cone-angles}. We find that 
		\begin{align*}
			p_+(x,m) & = \frac{m}{2} - \frac{3m^3-8m^2+8m}{4m-8} \frac{1}{x} + {O\left(\frac{1}{x^2}\right)} \\
			p_-(x,m) & = \frac{3m}{m-2} \frac{1}{x} + {O\left(\frac{1}{x^2}\right),} 
		\end{align*} 
		{where $O\left(\frac{1}{x^2}\right)$ denotes a function which is bounded in absolute value by a constant times $\frac{1}{x^2}.$} In particular, $\lim_{x \to + \infty} p_-(x,m) =0$ while $\lim_{x \to + \infty} p_+(x,m) = \frac{m}{2}$ as claimed.
		
		By Proposition \ref{proposition_admissible}, the scalar curvature of the corresponding metric is nowhere vanishing if and only if $p_{\pm}(x,m) \in (r_1,r_2)$. However, for large $x \gg 1$, we have
		$$r_1 = \frac{m}{2} - m \left( \frac{3m}{4} -1 \right) \frac{1}{x} + O\left(\frac{1}{x^2}\right).$$
		Hence, we find that for large $x \gg 1$ 
		$$p_+(x,m),{p_-(x,m)}<r_1$$
		and so the scalar curvature must vanish somewhere.
		
		We can formally take the limit as $x \to +\infty$ of the two $1$-parameter families of metrics in Corollary \ref{cor:m>2}. Replacing $p=p_{\pm}(x,m)$ into the formula for the coefficients of the polynomials ${\pol_{x,\pm}}$ given by equation (\ref{eq:q's for cone angle}) and taking the limit as $x \to +\infty$ gives the polynomials 
		\begin{align*}
			\tilde{\pol}_{+}(t) & =(t-1)(t+m-1) \\
			\tilde{\pol}_{-} (t) & = \frac{t-1}{3} \left( (m-2)t^2 + (m+1)t + m+1 \right).
		\end{align*}
		Then, as $\tilde{\pol}_+$ is a polynomial of degree two, it follows from by (i.b) in Theorem \ref{prop:Globalizing} that the resulting metric is complete with quartic volume growth. In addition, for such a metric we have $q_3=0=q_4$ as the terms of order $3$ and $4$ in $\tilde{\pol}_+(t)$ vanish. Hence, $s(a)=0$ and the resulting metrics coincides with the scalar-flat ALE ones from Corollary \ref{cor:scalar flat}.
		
		As for $\tilde{\pol}_-$, it is a polynomial of degree $3$ and the same Theorem \ref{prop:Globalizing} implies that it gives a complete metric with exponential volume growth. On the other hand, by looking at the coefficients of $\tilde{\pol}_-(t)$ we find that $a^2 q_0=\frac{m+1}{2}$, $q_1=0$, $\frac{aq_3}{6}=-\frac{m-2}{3}$, and $q_4=0$. Hence, we have that $s(a)=2q_3+q_4=- \frac{-4(m-2)}{a}$ which yields
		$$a s(a) = -4(m-2).$$
		Such a limiting metric corresponds to the negative K\"ahler--Einstein metric in the second bullet of Corollary \ref{cor:scalar flat}.
\end{proof}

\section{Einstein metrics}\label{sec:cke}
In this section we shall use the Bach-flat K\"ahler metrics we have constructed to study the Einstein metrics they give rise to through Derdzi\'nski's Theorem. 

\subsection{Cone angle metrics on $\mathbb{CP}^2 \# \overline{ \mathbb{CP}^2}$ and the Taub-bolt metric}
\subsubsection{Existence}
The Bach-flat K\"ahler metrics we have constructed in Proposition \ref{cor:m=1} have nowhere vanishing scalar curvature.

\begin{prop}\label{prop: Einstein cone angle}
	For any $\beta<4\pi$, there is an Einstein metric $g^\beta$ on $\mathbb{CP}^2 \# \overline{ \mathbb{CP}^2}$ with a cone angle $\beta$ along an $\Sigma\simeq S^2$ with self-intersection $1$. Such metrics are conformal to the Bach-flat K\"ahler metrics from Proposition \ref{cor:m=1}.
	
	In addition, as $\beta \to 0$, the metrics $g^{\beta}$ converge uniformly with all derivatives to a complete Ricci-flat metric on $(\mathbb{CP}^2 \# \overline{ \mathbb{CP}^2} ) \backslash \Sigma$ with cubic volume growth.
\end{prop}
\begin{proof}
		Let $x = \frac{b}{a} \geq 1$, and consider {$g_{x,\pm}$} the $U(2)$-invariant Bach-flat, cone angle K\"ahler metrics on $\mathbb{H}_1=\mathbb{CP}^2 \# \overline{ \mathbb{CP}^2}$ with weights {$p_{\pm}(x)$} along the divisor {$S_b=\Sigma$} given by Corollary \ref{cor:m=1}. Their scalar curvature {$\scal(g_{x,\pm})$} never vanishes and so the conformally related metrics {$(\scal(g_x))^{-2} g_x$} are Einstein by Derdzi\'nski's Theorem. Furthermore, because {$\scal(g_{x,\pm})$} is a smooth function, the metric {$(\scal(g_{x,\pm}))^{-2} g_{x,\pm}$} also has a cone angle along the {$\Sigma$}. Since $p_+(x)$ is decreasing and satisfies $p_+(x_0)=p_-(x_0)$ where $x_0$ is described in Corollary \ref{cor:m=1} {as the only positive root of $x^4+4x^3-14x^2-12x+9$}. Also $\lim_{x \to + \infty} p_+(x)=\frac{1}{2}$. On the other hand $p_-(x)$ is increasing and satisfies $\lim_{x \to 3^-}p_-(x)=+\infty$. This shows that all weights $p \in \left( \frac{1}{2} , +\infty \right)$ are achieved by the metrics {$g_{x,\pm}$} of Corollary \ref{cor:m=1} and so are by {$(\scal(g_{x,\pm}))^{-2} g_{x,\pm}$}. The same holds true for the conformally related Einstein metrics from Derdzi\'nki's theorem.
		
		Now we consider the limit of these metrics when $p \to + \infty$. We shall show that as $p \to +\infty$, the metrics converge smoothly on {the pre-image via $T_m$ composed with the moment map  of $$\{(x_1,x_2): x_1\geq0, \ x_2\geq0, a\leq\ x_1+x_2 <3a  \},$$ to a complete, Bach-flat K\"ahler metric of finite volume} $g$.  This metric arises from Theorem \ref{thm:Local existence and uniqueness} with $as(a)=6$, and is defined on the {pre-image via $T_m$ composed with the moment map  of 
		$$\{(x_1,x_2):x_1,x_2\geq 0, a\leq x_1+x_2<3a\}$$} 
 		and has positive scalar curvature.
		{Consider the formulas for the $q_i$ from equation (\ref{eq:q's for cone angle}) with $m=1$, $x=3$} and take the limit $p \to +\infty$. The limiting polynomial $\tilde{\pol}(t)=\frac{r^2-q(r)}{a^2}|_{r=at}$ given by
		$$\tilde{\pol}(t)=\frac{1}{8}(t+1)(t-1)(t-3)^2,$$
		which satisfies $q_0q_4=q_1q_3$. It yields a well defined Bach-flat K\"ahler metric for $t\in[1,3)$, i.e. $r \in [a,3a)$. We are in item's (ii.b) of Theorem \ref{prop:Globalizing} regime and so the resulting metric has a complete end as $r \to 3a$ and finite volume. Furthermore, Proposition \ref{proposition_admissible} guarantees that the scalar curvature $\scal(g)$ never vanishes. The scalar curvature along $S_a,$ can be computed from the explicit formulas for $q_3$, $q_4$ to give
		$$s(a)=2q_3+q_4 a=\frac{6}{a} .$$
		In fact, we have $a s(r)=9- \frac{3r}{a}$ which is positive for $r\in [a,3a).$

		As for the limiting scalar curvature, we find that it is given by the formula 
		$$ \scal(g)(r)= \frac{3}{a} \left(3- \frac{r}{a} \right),$$ 
		which is positive for $r\in [a,3a)$. 
		
		From Derdzi\'nski's Theorem the metric $\scal(g)^{-2} g$ is Einstein. Furthermore, its scalar curvature can be computed from the formula $S= 12 q_4^2 q_1 + 8 q_3^3 + 48 q_3 q_4 $ yielding $S=0$. Hence, the Einstein metric is actually Ricci-flat.
		
		Proceeding as in the proof of Theorem \ref{prop:Globalizing} we can consider a geodesic ray parameterised by $\ell \in (\ell_0 , +\infty)$ and given by $\ell \mapsto \curv(\ell)$ with $x_1(\curv(\ell))=\frac{\ell}{2}=x_1(\curv(\ell))$ and both $\theta_i(\curv(\ell))$ constant for $i=1,2 $. Here
		\begin{align*}
			\curv^* (\scal^{-2}g_{u_q}) = \scal^{-2}(\curv(\ell)) \frac{\ell}{ \pol(\ell)}   d\ell^2 =  \frac{1}{ 9 \left(3- \frac{\ell}{a} \right)^2 }  \frac{\ell}{ \pol(\ell)}   d\ell^2,
		\end{align*}
		and the length of the geodesic ray can be computed to be
		\begin{align}
			\mathrm{Length}(\curv) & = \int_{a}^{3a} \frac{1}{ 3 \left(3- \frac{\ell}{a} \right) }  \sqrt{\frac{\ell}{ \pol(\ell)}} d \ell =+\infty,
		\end{align}
		as
		\begin{align}
			 \int_{a}^{3a} \frac{d \ell}{\left( 3a- \ell \right)^2} =+\infty  .
		\end{align}
		The metric is therefore complete. We can also use this computation to find the geodesic length $u$ along $\curv$. This must satisfy $$du^2 = \frac{\ell d\ell^2}{\scal^2(\ell) \pol(\ell)} $$ 
		hence
		$$u \sim \frac{1}{3a - \ell}.$$
		On the other hand, the volume form for of the metric $(\scal{g})^{-2}g$ is given by
		$$ {\scal(g)}^{-4} dx_1 \wedge dx_2 \wedge d\theta_1 \wedge d\theta_2 , $$
		and changing coordinates to $r=x_1+x_2$, $R=x_2-x_1$ we can compute the volume of $B_\ell,$ the pre-image via $T_m$ composed with the moment map of the region 
		$$\{(x_1,x_2): x_1>0, \ x_2>0, \ a< x_1+x_2 <\ell  \}.$$ It is given by
		\begin{align*}
			\mathrm{Vol}(B_\ell)  = \frac{4\pi^2}{2} \int_a^\ell dr \int_{-r}^{r}  \frac{dR}{{\scal(g)}^{4}} = 2 \pi^2 \int_a^\ell   \frac{2r dr}{{\scal(g)}^{4}} 
		\end{align*}
		which is of the order of $\frac{1}{(3a-\ell)^3}$ i.e. $u^3$ showing that the resulting Einstein metric has cubic volume growth.
\end{proof}

\begin{remark}
	The Ricci-flat metric in $\CP^2\# \overline{\CP^2} \backslash \Sigma$ obtained as the limit $p\to +\infty$ in the previous Proposition is known as the smooth Taub-bolt metric (see \cite{BG}). The fact that our limiting metric must coincide with the smooth Taub-bolt metric follows from the fact that it is Ricci-flat and ALF toric as we show in appendix \ref{diagonalizing_metric} together with the uniqueness results in \cite{BG}. 
\end{remark}

\subsubsection{The Hitchin-Thorpe inequality}
The Hitchin-Thorpe inequality was used by LeBrun to construct manifolds not admitting Einstein metrics (see \cite{l}). It was shown by Atiyah-LeBrun (see \cite{al}) that this inequality can be generalised to the edge-cone angle setting. 
\begin{thm}[Atiyah-LeBrun]
	Let $M$ be a smooth 4-manifold and let $\Sigma$ be an embedded smooth curve in $M$. In $M$ admits and edge-cone metric with cone angle $2\pi\beta$ along $\Sigma$ then we have
	$$
	2\chi(M)\pm3\tau(M)\geq (1-\beta)(2\chi(\Sigma)\pm(1+\beta)[\Sigma]\cdot[\Sigma]),
	$$
	where $\chi$ denotes the Euler characteristic and $\tau$ denotes the signature.
\end{thm}
Edge-cone angle metrics were defined in \cite{al}. Using the theorem's notation, a metric on $M$ is said to be edge-cone with cone angle $2\pi\beta$ along $\Sigma$ if around each point of $\Sigma$ there are coordinates $(y^1,y^2,y^3,y^4)$ in which $\Sigma$ is given by $y^1=y^2=0$ and
$$
g=d\rho^2+\beta^2\rho^2(d\theta+a_3dy^3+a_4dy^4)^2+\sum_{i,j=3}^4a_{ij}dy_idy_j+\rho^{1+\epsilon}h
$$
where $y^1=\rho\cos(\theta)$ and $y^2=\rho\sin(\theta).$ Moreover $a_i$ and $a_{ij}$ are smooth functions and $h$ is a smooth $2$-tensor in the normal directions to $\Sigma.$

{These edge-cone metrics include K\"ahler cone angle metrics as defined by Donaldson in \cite{dca}}. Our metrics are conformal to K\"ahler cone angle metrics and are therefore also of Edge-cone type as multiplying by a non-vanishing conformal factor preserves the form for $g$ above. 
Therefore the above theorem applies to our setting in this paper. Let $M=\cp^2\# \overline{\cp^2}.$ This has Euler Characteristic $4$ and signature $0.$ In fact, the second cohomology is spanned by the fibre of $\mathbb{P}(\mathcal{O}(-m)\oplus \mathcal{O})\rightarrow \cp^1$ together with the projectivization of the section $(0,1)$ of  $\mathcal{O}(-m)\oplus \mathcal{O}\rightarrow \cp^1.$ The intersection pairing in this basis for the cohomology is given by
$$
\begin{pmatrix}
	0&1\\
	1&m\\
\end{pmatrix}.
$$
Assume that $T_m$ of the moment polytope of $M$ is $P_H.$ The pre-image of the edge $x_1+x_2=b$ via $T_m$ composed with the moment map defines the divisor $\Sigma=S_b$ in $M.$ This is a $\cp^1$ and therefore has Euler Characteristic $2.$ Its self-intersection can be calculated via toric geometry. In fact the interior normal to edge $x_1+x_2=b$ in $P_H$ is $(-1/m,-1/m).$ We can write it as a linear combination of the normals to adjacent edges. Namely
$$
(1,0)+(0,1)=-m(-1/m,-1/m).
$$ 
It follows that 
$$
[\Sigma]\cdot[\Sigma]=m.
$$
The Hitchin-Thorpe inequality becomes 
$$
8\geq(1-\beta)(4\pm m(1+\beta))
$$
We are interested in the case when $m=1$ as this is the only case when our Einstein metrics are smooth everywhere except at the the cone angle divisor. This equality reads 
$$
8\geq(1-\beta)(4\pm (1+\beta))
$$
which yields $0<\beta\leq 5$.

It is thus interesting to understand what cone angles are covered by our construction and compare with the Hitchin-Thorpe inequality. To do this we must understand the range of the cone angle arising in Proposition (\ref{cor:m=1}). From the proof of the proposition we know that the cone angles for which we construct Einstein metrics (which are the same cone angles as the conformally related extremal metrics)
are 
\begin{itemize}
	\item For $x_0<x$
	\[
	\scalebox{.9}
	{$p_+(x)=\frac{-3 +2 x-6x^2-6x^3 + x^4 + (1 + 4 x + x^2)\sqrt{9 - 12 x - 14 x^2 + 4 x^3 + x^4}}{2(-3 + x) (-2+3 x+3 x^2+2 x^3)}.$}
	\]
	This is actually defined at $x=3.$ It is decreasing for $x>x_0$ and has limit $1/2$ at infinity. Therefore this cone angle takes values in $(1/2,p_+(x_0)].$
	\item For $x_0<x<3$
	\[
	\scalebox{.9}
	{$p_-(x)=\frac{-3 +2 x-6x^2-6x^3 + x^4 - (1 + 4 x + x^2)\sqrt{9 - 12 x - 14 x^2 + 4 x^3 + x^4}}{2(-3 + x) (-2+3 x+3 x^2+2 x^3)}.$}
	\]
	This is increasing and tends to infinity at $3.$ Therefore this cone angle takes values in $[p_-(x_0),\infty).$ 
\end{itemize}
We can thus show the following proposition
\begin{prop}
	There is an Einstein edge-angle metrics on $\cp^2 \# \overline{\cp^2}$ with cone angle along the $(+1)$-sphere $S_b,$ for every cone angle in $(0,4\pi).$
\end{prop}
The Hitchin-Thorpe inequality forbids the existence of Einstein Edge-angle metrics with cone angle greater than $10\pi.$ The natural question is then 
\begin{question*}
	Are there Einstein edge-angle metrics on $\cp^2 \# \overline{\cp^2}$ with cone angle in $[4\pi,10\pi]$ along a $(+1)$-sphere?
\end{question*}

\subsection{Poincar\'e-Einstein metrics} 
\subsubsection{Existence} We have seen in Proposition \ref{cor:m>2} that for $m>3$ and $x$ large the extremal K\"ahler metrics we have constructed on $\bbH_m$ in the cohomology class determined by $x$ have curvature vanishing along a hypersurface as do the metrics from Corollary \ref{cor:(iii)}. The corresponding Einstein metrics are thus defined on certain open manifolds whose geometry we want to understand.

We start with a few definitions which are due to Penrose (see \cite{an}).
\begin{defn}
Let $N$ be a manifold such that $\bar{N}$ is a manifold with boundary. 
\begin{itemize}
\item A smooth function $v$ on $\bar{N}$ is said to {\it define} $\partial \bar{N}$ if $dv\ne 0,$ $v=0$ in $\partial \bar{N}$ and $v>0$ in $N$, the interior of $\bar{N}.$ 
\item A metric $h$ on $N$ is said to be {\it conformally compact} if there is a smooth defining function for $\partial \bar{N}$ on $\bar{N}$ such that $v^2h$ extends to a smooth metric on $\bar{N}.$ 
\item The conformal class of the restriction of $v^2h$ to $\partial \bar{N}$ is independent of the choice of the defining function $v$ and is called {\it conformal infinity}. 
\item The metric $h$ on $N$ is said to be a metric filling of the conformal class of the restriction of $v^2h.$
\item If additionally such a metric is Einstein then it is said to be of {\it Poincar\'e- Einstein} type.

\end{itemize}
\end{defn}
If a metric $g$ on $\bar{N}$ is smooth and $v$ is a smooth defining function for $\partial N$ then $h=v^{-2}g$ is known to be complete and it is conformally compact (see \cite{an} and references therein). 
An important problem in Riemannian geometry is to understand to which extent conformal classes of metrics can be filled in the above sense by Poincar\'e-Einstein metrics. Both the existence and the uniqueness questions regarding such fillings have been extensively studied. The existence question was introduced by Fefferman and studied by Graham (\cite{fg}), by Graham and Lee (\cite{gl}) by Anderson in several papers (\cite{an2} for instance) and more recently by Gursky and Sz\'ekeyihidi (\cite{gs}) to mention just a few. It plays an important role in Physics through the Ads/CFT correspondence. 

The metrics we shall construct yield examples of such fillings. We can extract from them complete conformally compact metrics that are in fact Poincar\'e- Einstein metrics. 

Consider one of the K\"ahler, Bach-flat metrics $g$ we construct in Proposition \ref{cor:m>2} with cone angle on $\bbH_m$ or the incomplete metrics on the total space of $\mO(-m)$ from Corollary \ref{cor:(iii)} and suppose that their scalar curvature ${\scal(g)}$ vanishes. Recall that this happens {along the pre-image via $T_m$ composed with the moment map of $$\{(x_1,x_2): x_1>0, \ x_2>0,  \ x_1 + x_2 = -\frac{2q_3}{q_4} \},$$ where $q_3,q_4$ are solutions of the system in equation (\ref{eq:q's for cone angle}).} We consider the polytope
$$
Q=\{(x_1,x_2): x_1>0, \ x_2>0, \ x_1+x_2 \geq a , \ x_1 + x_2 \leq  -\frac{2q_3}{q_4} \}.
$$
Let $\bar{N}$ the pre-image via $T_m$ composed with the moment map of $Q$ in $\bbH_m.$ This is a disk bundle inside $O(-m)\rightarrow \cp^1.$ The boundary $\partial \bar{N}$ of this disk bundle is in fact the manifold $S^3/\bbZ_m$ which is a Lens space. In this setting either the scalar curvature ${\scal(g)}$ or minus the scalar curvature $-{\scal(g)}$ gives a boundary defining function for $\partial \bar{N}$ and the metric
$$
{\scal(g)^{-2}g}
$$ 
is smooth, complete and conformally compact on $N$. In addition, it follows from Derdzi\'nski's Theorem \ref{Thmderd2} that it is Einstein. Recall from Proposition \ref{prop:Conformal Einstein} that in this case, the scalar curvature $S$ of $\scal(g)^{-2}g$ is constant and given by
\begin{align*}
	S & = 12 q_4^2 q_1 + 8 q_3^3 + 48 q_3 q_4.
\end{align*}
where $q_1,q_3,q_4$ are solutions of the system in equation (\ref{eq:q's for cone angle}) and we can write them explicitly in terms of our data. 

\begin{prop}\label{prop:noncompact Einstein metrics}
	The following assertions hold:
	\begin{itemize}
		\item[(a)] Let $m \in \N$, then, up to scaling, there is an explicit $1$-parameter family of $U(2)$-invariant Poincar\'e-Einstein metrics on a disk-bundle inside the total space of $\mathcal{O}(-m)$ over $\cp^1$. On such a disk bundle, the metrics are conformal to incomplete Bach-flat K\"ahler metrics on the total space of $\mathcal{O}(-m)$.
		
		\item[(b)] For each $m\geq 3$ there are two extra $1$-parameter families of Poincar\'e-Einstein metrics on $\mathcal{O}(-m)$ which are conformal to Bach-flat K\"ahler metrics extending over $\mathbb{H}_m$ with cone angles along the divisor at infinity $S_b$.
	\end{itemize}	
	The scalar curvature of all these metrics is negative.
\end{prop}
\begin{proof}
	The first family of metric comes applying Derdzi\'nski's Theorem \ref{Thmderd2} to the metrics from Corollary \ref{cor:(iii)}. Let $m \in \N$, $a s(a) \gg 1$ and $g$ be the Bach-flat K\"ahler metrics from Corollary \ref{cor:(iii)}. Corollary \ref{cor:(iii)} implies that for all sufficiently large $a s(a)$, the metric $g$ is defined {on the pre-image via $T_m$ composed with the moment map of $$
\{(x_1,x_2): x_1>0, \ x_2>0, \ x_1+x_2 \geq a  \}.
$$ }
and its scalar curvature vanishes along the hypersurface given by the equation
	$$\frac{r}{a}=  1 + \frac{4m}{a s(a) +4 (m-2)}.$$
	Then, Derdzi\'nski's Theorem \ref{Thmderd2} implies that ${\scal(g)^{-2}g}$ is Einstein. We have
	$$S= -\frac{(a s(a))^2}{a^3}\left( a s(a) + 6(m-2) \right) ,$$
	which is negative for large $a s(a)$.
	
	The other two families come from applying Derdzi\'nski's Theorem \ref{Thmderd2} to the metrics from Proposition \ref{cor:m>2}. The Hirzebruch surface $\bbH_m$ is endowed with a K\"ahler form whose moment polytope is 
		$$ \{(x_1,x_2): x_1,x_2\geq 0, a\leq x_1+x_2\leq b\}.$$
	Then, {for large enough $x=b/a$,} there are two $U(2)$-invariant, Bach-flat, extremal K\"ahler metrics $g_{x,\pm}$ on $\mathbb{H}_m$ with a cone angle $p_{\pm}(x,m)$ along $S_b$. For $x$ large, the scalar curvature of these metrics vanishes along a hypersurface $\lbrace x_1+x_2=c \rbrace$ inside the polytope as we proved in Proposition \ref{cor:m>2} since we are in the regime $p_{\pm}(x,m)<r_1.$ Furthermore, we can write the vanishing locus as $\lbrace x_1+x_2=c(x) \rbrace$ for some $c(x) \in (a,b)$. Hence, $\scal(g_{x,\pm})$ does not vanish on 
		$$ \{(x_1,x_2): x_1,x_2\geq 0, a\leq x_1+x_2 < c{(x)}\}.$$
	and Derdzi\'nski's Theorem \ref{Thmderd2} again implies that ${\scal(g_{x,\pm})^{-2}g_{x,\pm}}$ is Einstein on the pre-image via $T_m$ composed with the moment map of such a domain.
	
	As for the scalar curvature of these metrics, it is given by the quantity $12 q_4^2 q_1 + 8 q_3^3 + 48 q_3 q_4$ which is negative for $x \gg 1$.
\end{proof}

\begin{remark}
	It is perhaps interesting to note that in addition to the metrics from part (b) of Proposition \ref{prop:noncompact Einstein metrics} there are two extra Poincar\'e--Einstein which come from applying Derdzi\'nski's theorem to the other side of the hypersurface cut by $\lbrace \scal(g_{x,\pm}) =0 \rbrace$, i.e. to the domain where 
	$$\{(x_1,x_2): x_1,x_2\geq 0, c(x)<  x_1+x_2\leq b\}.$$ 
	These metrics share the same conformal infinity and so both can be regarded as filings of $S^3/\bbZ_m$ 
\end{remark}

Poincarr\'e-Einstein fillings of $S^3/\bbZ_m$ were constructed by Page and Pope in \cite{pp} in relation to the AdS/CFT duality. The setting in \cite{pp} is different from our and in particular the authors do not use action-angle coordinates. We claim that our metrics must in fact be isometric to theirs.
\begin{prop}
The metrics in item (b) of Proposition \ref{prop:noncompact Einstein metrics} are isometric to the Page-Pope metrics in \cite{pp}.
\end{prop}
\begin{proof}Sketch:
Both Einstein metrics must be conformal to Bach-flat $U(2)$-invariant metrics which are not scalar-flat. From Derdzi\'nski's results it is enough to show that, up to isometry, there is a unique extremal $U(2)$-invariant metrics on $\bbH_m$ with given volume and fixed cone angle at $S_b.$  This is a consequence of the following facts: 
\begin{itemize}
\item Any $U(2)$-invariant K\"ahler metric on the total space of a disk bundle over $S^2$ must arise from the Calabi ansatz. \\
\item By using Legendre duality, one can translate the complex coordinates from the Calabi ansatz into action-angle coordinates to get a symplectic potential. \\
\item The symplectic potential is the Guillemin potential plus a function depending solely on $x+y.$ This is because $x+y$ is the action counterpart of the norm of the fibre coordinate on the complex side. \\
\item Because the symplectic potential is the symplectic potential of an extremal metric it must satisfy an ODE and therefore is of the form. \\
$$h''(r)= -\frac{1}{r} + \frac{r}{\pol(r)},$$
	        for a degree $4$ polynomial $\pol$ whose quadratic term is zero, satisfying condition (\ref{eq:conditions for closing smoothly in terms of p2}).
\end{itemize}
The claim follows from the uniqueness part of our main result \ref{thm:Local existence and uniqueness}.	        
\end{proof}

\subsection{{Beyond} Derdzi\'nski's theorem}\label{sec:Not Derdzinsky}

We can find Einstein metrics which do not arise from Derdzi\'nski's theorem by taking a re-scaled limit of our Einstein metrics in the case when $a s(a) \to 0$. For $m \neq 2$ these metrics are not K\"ahler, but are conformally K\"ahler. 

Recall that as $y = a s(a)$ tends to zero, our $U(2)$-invariant Bach-flat K\"ahler metrics arising from Theorem \ref{thm:Local existence and uniqueness} tend to a scalar-flat metric to which Derdzi\'nski's theorem does not apply. In this section we will denote our $U(2)$-invariant Bach-flat K\"ahler metrics by $g_{(a,y)}$ to make the parameter dependence explicit. We have seen in subsection \ref{ss:cscK} that for $y = 0$ the metrics $g_{(a,0)}$ are scalar-flat. Hence, the limit of the metrics $\scal (g_{(a,y)})^{-2} g_{(a,y)}$ as $y \to 0$ is not well defined. To overcome this we consider the re-scaled conformal metrics
$$\gein_{(a,y)}:= (ay^{-1} \scal (g_{(a,y)}))^{-2} g_{(a,y)} ,$$
which, for $y \neq 0$ are again Einstein by Derdzi\'nski's theorem. This sequence of metrics converges as $y \to 0$. Indeed, using the formula in Lemma \ref{lem:Vanishinh of scalar curvature} we find that
$$ay^{-1} \scal (g_{(a,y)}) =  \frac{1}{4m} \left(  y +8 (m-1) - (y +4 (m-2)) \frac{r}{a}    \right).$$
Taking the limit as $y \to 0$ we obtain the conformal factor $\frac{1}{4m} \left(  8 (m-1) - 4 (m-2) \frac{r}{a}    \right)$ which does not vanish identically. Hence, where defined, the metric 
$$\gein_{(a,0)}:= \left(  \frac{1}{4m} \left(  8 (m-1) - 4 (m-2) \frac{r}{a}    \right)  \right)^{-2} g_{(a,0)}$$
is Einstein because it is the smooth limit of the Einstein metrics $\gein_{(a,y)}$ as $y \to 0$. 

Next, we focus on understanding the domain of definition of the metrics $\gein_{(a,0)}$ for the different values of $m$. 

\begin{itemize}
\item We start with $m=1.$ In this case the conformal factor is $\left(\frac{r}{a}\right)^{-2}$ which never vanishes in the domain of definition of $g_{(a,0)}$ which consists of $\mO(-m)$. We have $\gein_{(a,0)}= \varphi^{2}g_{(a,0)}$ where $\varphi=\frac{a}{r}$. Therefore the scalar curvature $S$ of $\gein_{(a,0)}$ is given by the formula
\begin{equation}
	S = \varphi^{-3} \left( 6 \Delta \varphi + \scal(g_{(a,0)}) \varphi \right) = 6  \varphi^{-3}  \Delta \varphi ,
\end{equation}
where $\Delta$ denotes the Hodge Laplacian of $g_{(a,0)}$ and we used $\scal (g_{(a,0)})=0$. Using Lemma \ref{lemma:Laplacian gamma} gives 
\begin{align*}
	S & = 6  \varphi^{-3}  \Delta \varphi \\
	& = 6  \varphi^{-3} \left( -2 \varphi ' \left(2-\frac{q'}{r}\right) - 2 \varphi'' \left(r-\frac{q}{r}\right) \right) \\
	& = \frac{6 r^3 }{a^3} \left( \frac{2a}{r^2} \left(2-\frac{q'}{r}\right) -  \frac{4a}{r^3} \left(r-\frac{q}{r}\right) \right) \\
	& = \frac{6 r^3 }{a^3} \left( - \frac{2a q'}{r^3}  +  \frac{4a q}{r^4} \right) \\
	& = \frac{12}{a^2} \left( -  q'  +  \frac{2 q}{r	} \right),\\
	& = \frac{12}{a},
\end{align*}
where the last equality comes from the fact that in this case we have $q(r)=ar$. In particular $S$ is positive.  By Myers theorem $\gein_{(a,0)}$ must either compactify smoothly or have an incomplete end as $r \to +\infty$. In fact, the second occurs as we can see by using a strategy similar to that in the proof Proposition \ref{prop: Einstein cone angle}. Namely, we start by showing that the metric has infinite volume and so cannot be compact, then we compute the length of a curve running towards $r \to +\infty$ and find that this is finite, thus showing that infinity is at finite distance.

\item In the case $m=2,$  $\gein_{(a,0)}=g_{(a,0)}$.

\item Finally, we consider the case $m \geq 3$. In this situation the conformal factor diverges at 
$$\frac{r}{a}=\frac{2(m-1)}{m-2}>1,$$
so that metric $\gein_{(a,0)}$ has a Poincar\'e type end as $\frac{r}{a} \to \left(\frac{2(m-1)}{m-2}\right)^-$. Set 
$$\varphi(r)=\frac{4m}{ 8 (m-1) - 4 (m-2) \frac{r}{a}},$$
so that $\gein_{(a,0)}=\varphi^2 g_{(a,0)}.$ The scalar curvature $S$ of the Einstein metric $\gein_{(a,0)}$ is given by
\begin{align*}
	S & = 6  \varphi^{-3} \left( -2 \varphi ' \left(2-\frac{q'}{r}\right) - 2 \varphi'' \left(r-\frac{q}{r}\right) \right) .
\end{align*}
Inserting $q(r)=(m-1)a^2- (m-2) ar$ we find that
$$S=-\frac{12}{a}(m-2),$$
which is negative. In particular, this bullet proves the result stated in Theorem \ref{thm:Not Derdzinsky}.
\end{itemize}
It is interesting to note that these Einstein metrics are conformally K\"ahler. However, the K\"ahler metric to which they are conformal is scalar flat and so they cannot be obtained from directly applying Derdzi\'nski's theorem.

\appendix

\section{Convexity on $\mathbb{H}_m$}\label{appendix:positive_p}\label{appendix:convex}

\begin{prop}
	Let $m \in \N$, $p >0$, $b>a>0$ and 
	$$h''(r)=-\frac{1}{r}+\frac{r}{\pol(r)}$$
	for $\pol(r)=r^2-q(r)$ and $q(r)=\sum_{i=1}^4 \frac{q_i}{i!} r^i$ with $q_2=0$ and the remaining $q_0$, $q_1$, $q_3$ and $q_4$ satisfying the conditions \ref{eq:conditions for closing smoothly}. Then, the symplectic potential
	$$
		u=\frac{1}{2}\left(x_1\log(x_1)+x_2\log(x_2)+h(x_1+x_2)\right).
	$$
	is convex i.e. $\Hess(u)$ is positive definite on $P_H.$
\end{prop}
\begin{proof}
	It follows from Lemma \ref{lem:Convexity} that it is enough to prove that $\pol(r)$ is positive for $r \in ]a,b[$.
	
	Recall that the smoothness conditions \ref{eq:conditions for closing smoothly} amount to
	\begin{IEEEeqnarray}{lCr} \nonumber
		\pol(a)=\pol(b)=0 \\ \nonumber
		\pol'(a)=am, \quad \pol'(b)=\frac{-bm}{p}, \nonumber
	\end{IEEEeqnarray}
	from what follows that $\pol(r)=\frac{q_4}{24}(r-a)(b-r)(r^2+\alpha r+\beta)$ with
	$$
	\begin{cases}
		a^2+\alpha a+\beta=\frac{am}{\frac{q_4}{24}(b-a)}\\
		b^2+\alpha b+\beta=\frac{bm}{\frac{q_4}{24} p(b-a)}
	\end{cases}
	$$
	so that
	\begin{IEEEeqnarray}{lCr} \label{alphabeta}
		\alpha=\frac{(b/p-a)m}{\frac{q_4}{24}(b-a)^2}-(a+b)\\ \nonumber
		\beta=ab\left(1+\frac{m(p-1)}{\frac{q_4}{24}p(b-a)^2}\right)
	\end{IEEEeqnarray}
	Replacing $\frac{q_4}{24}$ with its value obtained from \ref{eq:conditions for closing smoothly} we get
	\begin{IEEEeqnarray}{lCr} \nonumber
		\alpha=a(x-1)\frac{mx+p(x^2+mx-1)}{mx(x+2)-p(x^2+2(m-1)x+m+1)}\\ \nonumber
		\beta=a^2x \frac{-m(2x+1)+p((m-1)x^2+(2m+1)x-1)}{mx(x+2)-p(x^2+2(m-1)x+m+1)}
	\end{IEEEeqnarray}
	Note that setting $Q(r)=\frac{q_4}{24}(r^2+\alpha r+\beta)$ we must show that $$Q(r)>0, \quad \forall r\in (a,b).$$ Now
	\begin{IEEEeqnarray}{lCr} \nonumber
		p'(a)=am=(b-a)Q(a) \\ \nonumber
		p'(b)=-\frac{bm}{p}=-(b-a)Q(b) \nonumber
	\end{IEEEeqnarray}
	so that $Q$ is positive at $a$ and $b.$ The extremum of $Q$ is attained at $-\alpha/2$ and is $-\frac{q_4}{24}(\alpha^2/4-\beta).$ So $Q$ can only be negative on a subset of $(a,b)$ if
	$$
	\begin{cases}
		a<-\frac{\alpha}{2}<b\\
		-\frac{q_4}{24}(\alpha^2/4-\beta)<0
	\end{cases}
	$$
	Let us start with the first condition. Using equation (\ref{alphabeta}) it becomes
	$$
	\begin{cases}
		\frac{mx(2x^2+5x-1)-p(x^3+3(m-1)x^2+3(m+1)x-1)}{mx(x+2)-p(x^2+2(m-1)x+m+1)}>0\\
		\frac{-3mx(x+1)-p(x^3+(m-3)x^2-(5m-3)x-(2m+1)))}{mx(x+2)-p(x^2+2(m-1)x+m+1)}>0.
	\end{cases}
	$$
	For fixed $m$ and $x$ we will analyse for which values of $p$ the above inequalities hold. The relevant quantities are
	$$
	\begin{aligned}
		&N_1=-3mx(x+1)-p(x^3+(m-3)x^2-(5m-3)x-(2m+1))),\\
		&N_2=mx(2x^2+5x-1)-p(x^3+3(m-1)x^2+3(m+1)x-1),\\ 
		&\mD=mx(x+2)-p(x^2+2(m-1)x+m+1).
	\end{aligned}
	$$
	which vanish respectively when $p$ equals
	$$
	\begin{aligned}
		&r_1=\frac{-3(x+1)mx}{x^3+(m-3)x^2-(5m-3)x-(2m+1)}, \\
		&r_2=\frac{(2x^2+5x-1)mx}{x^3+3(m-1)x^2+3(m+1)x-1},\\
		& d=\frac{(x+2)mx}{x^2+2(m-1)x+m+1}.
	\end{aligned}
	$$
	The relative positions of these quantities matter. The first thing to note is that 
	$$
	r_1>0 \iff x^3+(m-3)x^2-(5m-3)-(2m+1)<0.
	$$
	Note also that we alway have $d\leq r_2$ and $d\leq r_1,$ whenever $r_1>0.$ The first inequality comes from
	\begin{IEEEeqnarray}{lCr} \nonumber
		(2x^2+5x-1)(x^2+2(m-1)x+m+1)-(x^3+3(m-1)x^2+3(m+1)-1)(x+2)\\ \nonumber
		=(x^3+3x^2-3x+1)(x+m-1)>0
	\end{IEEEeqnarray}
	for all $x>1$ and positive (integer) $m.$ As for the second inequality  $d\leq r_1,$ whenever $r_1>0,$ it follows from
	\begin{IEEEeqnarray}{lCr} \nonumber
		3(x+1)(x^2+2(m-1)x+m+1)+(x+2)(x^3+(m-3)x^2-(5m-3)-(2m+1))\\ \nonumber
		=(x^3+3x^2-3x+1)(x+m-1)>0.
	\end{IEEEeqnarray}
	There are therefore 3 cases to consider.
	\begin{itemize}
		\item The case when $r_1<0.$ We have $0<d<r_2$
		\item The case when $r_1>0$ and $r_2<r_1.$ We have $0<d<r_2<r_1$
		\item the case when $r_1>0$ and $r_1<r_2.$ We have $0<d<r_1<r_2.$
	\end{itemize}
	Let us now look at the second condition $-\frac{q_4}{24}(\alpha^2/4-\beta)<0$. A tedious calculation using (\ref{alphabeta}) shows that 
	$$
	\alpha^2-4\beta=\frac{\mA p^2+2mx\mB p+m^2x^2\mC}{\left(mx(x+2)-p(x^2+2(m-1)x+m+1)\right)^2},
	$$ 
	where
	\begin{IEEEeqnarray}{lCr} \label{mABC}
		\mA=(x^3+3(m-1)x^2+3(m+1)x-1)^2 \\ \nonumber
		\mB= 3(x+1)(x-1)^3-m(2x^4+7x^3+18x^2+7x+2)\\ \nonumber
		\mC=9(x+1)^2.
	\end{IEEEeqnarray}
	Now we think of $\alpha^2-4\beta$ a quadratic function of $p$. The discriminant of this function is $4m^2x^2(\mB^2-\mA\mC).$ 
	\begin{itemize}
		\item If this is negative then as $\mA>0,$ $\alpha^2-4\beta>0$ always.
		\item Assume it is negative and $\mB>0$ then the roots of $\alpha^2-4\beta$ in $p$ are negative so for positive $p,$ $\alpha^2-4\beta>0.$
		\item Assume now that the discriminant is negative and $\mB<0.$ For fixed $m,x$ the roots of $\alpha^2-4\beta$ are
		$$
		\rho_1=mx\frac{-\mB-\sqrt{\mB^2-\mA\mC}}{\mA}, \quad \rho_2=mx\frac{-\mB+\sqrt{\mB^2-\mA\mC}}{\mA}.
		$$
		Because $\mB<0,$ both roots $\rho_1<\rho_2$ are positive and therefore we are interested in comparing them with $d, r_1$ and $r_2$. In fact we have $\rho_2<d.$ To see this note that this is equivalent to
		$$
		\begin{aligned}
			&\frac{-\mB+\sqrt{\mB^2-\mA\mC}}{\mA}<\frac{(x+2)}{x^2+2(m-1)x+m+1} \\
			&\iff (x^2+2(m-1)x+m+1)\left(-\mB+\sqrt{\mB^2-\mA\mC}\right)<(x+2)\mA \\
			&\iff (x^2+2(m-1)x+m+1)\sqrt{\mB^2-\mA\mC}<(x+2)\mA+(x^2+2(m-1)x+m+1)\mB\\
			&\iff (x^2+2(m-1)x+m+1)^2\left(\mB^2-\mA\mC\right)<\left((x+2)\mA+(x^2+2(m-1)x+m+1)\mB\right)^2 \\
			&\iff (x+2)^2\mA^2+2(x+2)(x^2+2(m-1)x+m+1)\mB\mA+(x^2+2(m-1)x+m+1)^2\mA\mC>0 \\
			&\iff (x+2)^2\mA+2(x+2)(x^2+2(m-1)x+m+1)\mB+(x^2+2(m-1)x+m+1)^2\mC>0 \\
		\end{aligned}
		$$
		Replacing in the formulas for $\mA, \mB$ and $\mC$ from equation (\ref{mABC}) it follows that
		$$
		(x+2)^2\mA+2(x+2)(x^2+2(m-1)x+m+1)\mB+(x^2+2(m-1)x+m+1)^2\mC=(x^3+3x^2-3x-1)^2(x+m-1)^2
		$$
		which is clearly positive for $x>1$ and positive $m.$ Thus $\rho_2<d.$ Note that in this case $\alpha^2-4\beta$ is positive everywhere except between $\rho_1$ and $\rho_2.$
	\end{itemize}
	Last we note that $\mD$ is negative when $p>d$ and positive if $p<d.$ Note also that $\mD$ and $q_4$ have the same sign.
	We are now ready to sum up our findings on 3 tables.
	\begin{itemize}
		\item Consider the case when $r_1<0.$ 
		
		Let us further assume for now that $\rho_1$ and $\rho_2$ are real positive roots (i.e. $\mB^2-\mA\mC>0$ and $\mB<0$.)The signs of the relevant quantities are as follows
		\begin{center}
			\begin{tabular}{|| c c c c c c c c c c c } 
				\hline
				$p$&0& &$\rho_1$& &$\rho_2$& &$d$& &$r_2$&  \\ 
				\hline
				$\mD$& &+ & & + & &+ &0 &-& &-  \\ 
				\hline
				$N_1$& &- & & - & &- & &-& &-  \\
				\hline
				$N_2$& &+ & & + & &+ & &+&0 &- \\
				\hline
				$\alpha^2-4\beta$& &+&0&-&0&+& & +& &+\\
				\hline
				$N_1/\mD$& &- & & - & &- &$\times $&+& &+ \\
				\hline
				$N_2/\mD$& &+ & & + & &+ &$\times $&-&0 &+ \\
				\hline
				$D(\alpha^2-4\beta)$& &+&0&-&0&+&0 & -& &-\\
				\hline\hline
			\end{tabular}
		\end{center}
		\
		Because $N_1/\mD, N_2/\mD, \mD(\alpha^2-4\beta)$ are never positive for the values of $p$ we see that in this case $p(r)$ is positive for $r in (a,b)$.

		Still under the assumption that  $r_1<0$ but assuming that either $\rho_1$ and $\rho_2$ are negative or complex valued we can still use the above sign table by simply discarding what is on the left of $\rho_2$ and $\rho_2.$
		\item Let us now consider the case when $r_1>0$ and $r_1<r_2.$ 
		
		Again let us further assume that $\rho_1$ and $\rho_2$ are real positive roots (i.e. $\mB^2-\mA\mC>0$ and $\mB<0$.)The signs of the relevant quantities are as follows
		\begin{center}
			\begin{tabular}{|| c c c c c c c c c c c c c } 
				\hline
				$p$&0& &$\rho_1$& &$\rho_2$& &$d$& &$r_1$& &$r_2$&  \\ 
				\hline
				$\mD$& &+ & & + & &+ &0 &-& &-& &-  \\ 
				\hline
				$N_1$& &- & & - & &- & &-&0&+& &+  \\
				\hline
				$N_2$& &+ & & + & &+ & &+& &+&0 &- \\
				\hline
				$\alpha^2-4\beta$& &+&0&-&0&+& & +& &+& &+\\
				\hline
				$N_1/\mD$& &- & & - & &- &$\times $&+&0&-& &- \\
				\hline
				$N_2/\mD$& &+ & & + & &+ &$\times $&-& &-&0 &+ \\
				\hline
				$D(\alpha^2-4\beta)$& &+&0&-&0&+&0&-& & -& &-\\
				\hline\hline
			\end{tabular}
		\end{center}
		\
		The 3 last rows are never simultaneously positive therefor $p(r)$ is positive for $r \in (a,b).$ Again by ignoring what is to the left of $\rho_2$ we see that the same extends to when $\rho_1$ and $\rho_2$ are not real positive roots of $\alpha^2-4\beta.$
		\item The case when $r_1>0$ and $r_2<r_1$ is similar. The details are in the table below.
		\begin{center}
			\begin{tabular}{|| c c c c c c c c c c c c c } 
				\hline
				$p$&0& &$\rho_1$& &$\rho_2$& &$d$& &$r_2$& &$r_1$&  \\ 
				\hline
				$\mD$& &+ & & + & &+ &0 &-& &-& &-  \\ 
				\hline
				$N_1$& &- & & - & &- & &-& &-&0 &+  \\
				\hline
				$N_2$& &+ & & + & &+ & &+&0 &-& &- \\
				\hline
				$\alpha^2-4\beta$& &+&0&-&0&+& & +& &+& &+\\
				\hline
				$N_1/\mD$& &- & & - & &- &$\times $&+& &+& 0&- \\
				\hline
				$N_2/\mD$& &+ & & + & &+ &$\times $&-&0 &+& &+ \\
				\hline
				$D(\alpha^2-4\beta)$& &+&0&-&0&+&0&-& & -& &-\\
				\hline\hline
			\end{tabular}
		\end{center}
		\
		Again, the 3 last rows are never simultaneously positive, therefore $p(r)$ is positive for $r \in (a,b).$ By ignoring what is to the left of $\rho_2$ we see that the same extends to when $\rho_1$ and $\rho_2$ are not real positive roots of $\alpha^2-4\beta.$
	\end{itemize}
\end{proof}

\section{Metric classification}\label{appendix:mc}

We would like to understand how the metrics we have constructed behave when the parameter $y=as(a)$ varies. We have seen in Proposition (\ref{prop:Globalizing}) that there are $4$ kinds of $U(2)$-invariant Bach-flat K\"ahler metrics we can obtain from using Calabi's ansatz. Their behaviour depends on the value of $as(a)$ through the behaviour of the polynomial $\pol$. We will assume that $m\geq 3.$ 
The relevant polynomial is
\begin{align*}
		\frac{\tilde{\pol}(t)}{t-1} & =  \frac {( y+12m )  ( y +8(m-1) )}{96m}  - \frac { \left( y+12m \right)  \left( y-8 \right) }{ 96m} t \\ 
		& \ \ \ \ -{\frac {y \left( 12m+y-8 \right) }{96m}}{t}^{2}+{\frac {y \left( y+4(m-2) \right) }{96m}}{t}^{3}.
	\end{align*}
Set $\tau=t-1.$ In terms of $\tau$ the above becomes $\frac{1}{96m}$ times
$$
\pol_{m,y}(\tau)=y(y+4(m-2))\tau^3+2y(y-8)\tau^2+24m(4-y)\tau+96m^2.
$$

\begin{itemize}
\item When this polynomial has degree is $3$ and no positive roots the metric we constructed is defined on $\mO(-m)$ and it is incomplete. \\
\item When the polynomial is of degree $3$ and has a positive root $\tau_1,$ then the metric is defined on $\bbH_m$ and it has a cone angle $\frac{-ma(1+\tau_1)}{\pol'(a(1+\tau_1))}$.\\
\item When the polynomial is of degree $2$ then the metric we constructed is defined on $\mO(-m),$ its complete and has exponencial volume growth . \\
\item When the polynomial is of degree $1$ then the metric we constructed is defined on $\mO(-m),$ its complete and has quartic volume growth . 
\end{itemize}

The discriminant of this polynomial is  $-768{m}^{2}$ times
$$
D=y^3\left( 6(m-2)+y \right)\left( 12m+y \right) \left( y^3+6(3m-2)y^2+72m(m-2)y+256\right).
$$
When $D>0,$ the polynomial will have one root whereas if $D<0$ it will have three roots. Let 
$$
v(x)= y^3+6(3m-2)y^2+72m(m-2)y+256.
$$
Then $v(-12m)=256=v(-6(m-2))=v(0)$ and $v(-4(m-2))=-64m^2(m-3)<0.$ This means that
\begin{itemize}
\item The polynomial $v$ admits a zero in $]-\infty, -12m[$ say $y_1,$ \\
\item it admits another zero in $]-6(m-2),-4(m-2)[$ say $y_2,$ \\
\item and yet another zero in $]-4(m-2),0[,$ $y_3.$
\end{itemize}
The sign table for $D$ and the highest order term $y(y+4(m-2))$ of our polynomial is

\begin{center}
\begin{table}[htbp]
    \addtolength{\tabcolsep}{-3pt}
             \begin{tabular}{|| c |c c c c c c c c c c c c c c c c c } 
				\hline
				$y$&$-\infty$& &$y_1$& &\tiny{$-12m$}& &\tiny{$-6(m-2)$}& &$y_2$& &\tiny{$-4(m-2)$}& & $y_3$& & $0$& &\\
				\hline
				\tiny{$6(m-2)+y$} & & -& & -& & -&0&+ & & +& & +& & +&  & +& \\
				\hline
				\tiny{$12m+y$} & & -& & -&0 & +& & +& & +& & +& & +& &+& \\
				\hline
				$v$& & -&$0$&+ & &+ & & +&0& -& &- & 0&+ & &+& \\
				\hline
				$D$& & -&$0$&+& 0&-& 0& +&0& -& &- & 0&+ &0 &-& \\
				\hline
				\tiny{${y(y+4(m-2))}$}& & +& &+& & +& & +& & +&0 &- & &- &0 &+& \\
				
				\hline\hline
			\end{tabular}
			\end{table}
		\end{center}

Now $\pol_{m,y}$ is positive at $0$ so if its highest order term is positive and it has a single root, that root is negative.This means that for $y$ in $]-\infty,y_1[$,  $]-12m,-6(m-2)[$, $]y_2,-4(m-2)[$ or $]0,+\infty[$ our metric is defined over $\mO(-m)$. 

On the other hand if its highest order term is negative and it has a single root, then the root is positive. This means for $y$ in $]-4(m-2),y_3[$ our metric is defined on $\bbH_m.$

Let us now consider the case when $\pol_{m,y}$ has $3$ positive roots and we are outside $]-4(m-2),0[$. This concerns the intervals $]y_1,-12m[$ and $]-6(m-2),y_2[.$ The product of the 3 roots has the same sign as $-y(y+4(m-2))$ since the constant term for $\pol_{m,y}$ is positive. This means that outside $]-4(m-2),0[,$ the product of the 3 roots of $\pol_{m,y}$ is negative. So either all 3 roots are negative or 2 are positive and one is negative. Assume the latter. This would means that $\pol'_{m,y}$ has at least a positive root. But 
$$
\pol'_{m,y}(\tau)=3y(y+4(m-2))\tau^2+4y(y-8)\tau+24m(4-y).
$$
Therefore since in our regime $y(y+4(m-2))>0$ the product of the two roots of $\pol'_{m,y}$ has the same sign as $4-y$ i.e. it is positive and the two roots of $\pol'_{m,y}$ are positive. But their sum has the same sign as $-4y(y-8)$ which is negative as $y<0.$ We conclude that in $]y_1,-12m[$ and $]-6(m-2),y_2[,$ $\pol_{m,y}$ has no positive roots and our metric is defined on $\mO(-m)$ and is incomplete. 

We are now left with analysing the metric for parameter values in $]-4(m-2),0[.$ In $]-4(m-2),y_1[,$ $\pol_{m,y}$ has a single root and since $\pol_{m,y}(0)>0$ and $\pol_{m,y}$ tends to $-\infty$ at $+\infty$ then the root is positive so that the metric is defined on $\bbH_m.$ In $]y_1,0[.$ $\pol_{m,y}$ has 3 roots. At least one root is positive because $\pol_{m,y}(0)>0$ and $\pol_{m,y}$ tends to $-\infty$ at $\infty$ so again the metric is defined on $\bbH_m.$

We have seen that if $y+4(m-2)>0$ then the scalar curvature of $U(2)$-invariant Bach-flat K\"ahler metrics obtained from Proposition (\ref{prop:Globalizing}), the scalar curvature of such metrics vanishes along an $S_b$ where 
$$
b=a\left(1+\frac{4m}{y+4(m-2)}\right).
$$
In this case we get Einstein metrics from Derdzi\'nski's theorem on an open set of the whole manifold where the scalar curvature is either positive or negative. The scalar curvature of this Einstein metric is given 
$$
S=\frac{-y^2\left(y+6(m-2)\right)}{a^3}.
$$
In the regime where $y+4(m-2)<0$ the scalar curvature of the constructed $U(2)$-invariant Bach-flat K\"ahler metrics is nowhere vanishing and so we get Einstein metrics on $\mO(-m)$ whose scalar curvature takes on all constant values smaller than 
$$
\frac{32(m-2)^3}{a^3}.
$$
These Einstein metric are incomplete as one can see by slightly modifying the arguments in the proof of Proposition (\ref{prop:Globalizing}). Namely Equation \ref{incomplete}. We give the details. When we consider the geodesic ray parameterised by $\ell \in (\ell_0 , +\infty)$ and given by $\ell \mapsto \curv(\ell)$ with $x_1(\curv(\ell))=\frac{\ell}{2}=x_{{2}}(\curv(\ell))$ and both $\theta_i(\curv(\ell))$ constant for $i=1,2 $ we find that for the Einstein metric 
		\begin{align}
		\mathrm{Length}(\curv) & = \int_{\ell_0}^{+\infty}  \sqrt{\frac{\ell}{(q_3+2q_4 \ell)^2 \pol(\ell)}} d \ell \sim \int_{\ell_0}^{+\infty} \ell^{-\frac{\alpha+1}{2}} d \ell ,
	\end{align}
	which is always finite. The metric is incomplete. 
	
Note also that for the parameter choice $y=-6(m-2)$ the incomplete Einstein metric is actually Ricci-flat. 

We now focus on the values $y=-4(m-2)$ and $y=0$. The degree of $\pol_{m,y}$ at each of these is $2$ or $1$ so that it follows from Proposition (\ref{prop:Globalizing}) that the corresponding $U(2)$-invariant Bach-flat K\"ahler metrics have exponencial or quartic volume growth. We have also seen that when $y=-4(m-2)$ the metric has constant scalar curvature equal to $s(a)$ and so it is Einstein whereas for $y=0$ the metric is scalar-flat. In particular we cannot apply Derdzi\'nski's construction to it. 

We can determine the cone angles at infinity for these metrics by using the formulas 
\[
\scalebox{.8}{$p_-(x)=\frac{m}{2}\frac{m (x^4 + 2 x^3 + 6 x^2 + 2 x + 1)-2 (x + 1) (x - 1)^3-(1 + 4 x + x^2) \sqrt{-4 m (-1 + x)^3 (1 + x) + 4 (-1 + x^2)^2+m^2 (1 - 4 x - 6 x^2 - 4 x^3 + x^4)}}{(2 + m - 2 x + m x) (-2 + 3 m x + 3 m x^2 + 2 x^3)}$}
\]
\[
\scalebox{.8}{$p_+(x)=\frac{m}{2}\frac{m (x^4 + 2 x^3 + 6 x^2 + 2 x + 1)-2 (x + 1) (x - 1)^3+(1 + 4 x + x^2) \sqrt{-4 m (-1 + x)^3 (1 + x) + 4 (-1 + x^2)^2+m^2 (1 - 4 x - 6 x^2 - 4 x^3 + x^4)}}{(2 + m - 2 x + m x) (-2 + 3 m x + 3 m x^2 + 2 x^3)}$}
\]
which leads to
$$
\begin{aligned}
p_-(x)&=&\frac{m}{(m-2)x}+O\left(\frac{1}{x^2}\right),\\
p_+(x)&=&\frac{m}{2}+\frac{m(-3m^2+8m-16)}{4(m-2)x}+O\left(\frac{1}{x^2}\right).\\
\end{aligned}
$$
Now $p_+$ corresponds to $y=0$ while $p_-$ corresponds to $y=-4(m-2)$, therefore the limiting metrics have cone angles $m/2$ and $0$ at infinity respectively. We sum up findings in the following table

 \begin{center}
    \begin{table}[htbp]
    \addtolength{\tabcolsep}{-6pt}
\tiny
			\begin{tabular}{||c|ccccccccccccccccc } 
			        \hline
				$y$&$-\infty$& &$y_1$& &\tiny{$-12m$}& &\tiny{$-6(m-2)$}& &$y_2$& &\tiny{$-4(m-2)$}& & $y_3$& & $0$& &\\
			        \hline
				Space& &$\mO(-m)$ & & $\mO(-m)$& & $\mO(-m)$& & $\mO(-m)$& & $\mO(-m)$& &$\bbH_m$ & &$\bbH_m$  & &$\mO(-m)$ & \\ 
				\hline
				Type & & Inc& & Inc&  & Inc& & Inc& & Inc& cscK& & & & $sf$ &Inc& \\
				\hline
				\tiny{ $\scal=0$} & & na& & na&  & na& & na& & na& &$S^3/\bbZ_m$ & &$S^3/\bbZ_m$ & $ $ &$S^3/\bbZ_m$& \\
				\hline
				\tiny{ Einstein Type} & & Inc& & Inc&  & Inc& & Inc& & Inc& &2PE & &2PE & $na$ &2PE& \\
				\hline
				{\tiny Einstein scalar} & &- & &- &&- &Rf & +&max &+ & &+ & &+ &  &+& \\
                                 \hline
			\end{tabular}
	\end{table}
	\end{center}	
Here Inc stands for incomplete, while PE stands for Poincarr\'e-Einstein, sf for scalar-flat and Rf for Ricci-flat. We write 2PE because we get 2 Poincarr\'e-Einstein metrics from our constructions a smooth PE and a cap which is either incomplete or has a cone angle. 

\section{Derdzi\'nski's Theorem}\label{d}

For completeness, in this appendix we give a proof of Derdzi\'nski's theorem following Derdzi\'nski's original work \cite{d} as well as some more recent developments by LeBrun \cite{LeBrun_sd}. We start by recalling the statement of Derdzi\'nski's theorem.

\begin{theorem}\label{thm:Derdzinsky}
	Let $(M,g)$ be an oriented Einstein manifold with $\det W^+(g)>0$ where $W^+(g)$ as an endomorphism of $\Lambda^2_+$. Then, the metric 
	$$h= (24)^{1/3} |W^+(g)|^{2/3}g$$ 
	is K\"ahler and extremal. Furthermore, if the scalar curvature $\scal(h)$ of $h$ is nowhere vanishing, $g=(\scal(h))^{-2}h$. In particular, $\scal(h) = (24)^{1/3} |W^+|^{2/3}$.
	
	Conversely, suppose that $(M,h)$ is an extremal K\"ahler manifold whose scalar curvature $\scal(h)$ satisfies the equation
	$$\scal(h)^{3} \left( 6 \Delta_h \scal(h)^{-1}  + 1 \right) = \text{constant}.$$
	Then $g=(\scal(h))^{-2}h$ is Einstein where defined, with scalar curvature $\scal(g)=\scal(h)^{3} \left( 6 \Delta_h \scal(h)^{-1}  + 1 \right) $.
\end{theorem}

We shall start by proving a series of lemma's.

\begin{lemma}[LeBrun in \cite{LeBrun_sd}]
	Let $(M,g)$ be a simply connected $4$-manifold whose positive Weyl tensor $W^+$ is harmonic and satisfies $\det W^+ >0$. Then, the metric $h=|W^+|^{2/3}g$ is K\"ahler, and has positive scalar curvature equal to
	$$\scal(h) = 6 \left(\frac{2}{3}\right)^{1/6} |W^+|^{1/3} , $$
	and $g=(\frac{2}{3})^{1/3} 6^2 (\scal(h))^{-2}h$.
\end{lemma}
\begin{proof}
	Given that $W^+$ is trace-free, when viewed as an endomorphism of $\Lambda^2_+$, we find from the hypothesis $\det W^+ >0$, that it must have one positive and two negative eigenvalues. Furthermore, as $M$ is simply connected, the positive eigenspace of $W^+$ sweeps out a trivial real line sub-bundle of $\Lambda^2_+$. 
	
	Now, let $\alpha:M \to \mathbb{R}^+$ denote the positive eigenvalue of $W^+$ and $\omega \in \Omega^2_+$ an associated eigenform with $|\omega|_g^2=2$. A computation carried out in Theorem 2.1 of \cite{LeBrun_sd} using the Weitzenb\"ock formulas for the self-dual $2$-form $\omega$ and the self-dual Weyl tensor $W^+$ shows that $\omega$ is parallel with respect to the Levi-Civita connection associated with $\alpha^{2/3}g$. As a consequence, $h=\alpha^{2/3}g$ is a K\"ahler metric which, by conformal invariance of $W$, has the same self-dual Weyl tensor as $g$. However, in order to view the self-dual Weyl tensor as an endomorphism of $\Lambda^2_+$ requires raising an index and and as such $W^+_g$ and $W^+_h$ differ in general. Hence $W_g^+=\alpha^{2/3} W^+_h$. Moreover, for any K\"ahler metric such as $h$, we have
	$$W_h^+= \mathrm{diag} \left( \frac{\scal(h)}{6} , -\frac{\scal(h)}{12} , \frac{\scal(h)}{12} \right),$$
	and so, equating the first eigenvalue of $W_g^+$ and that of $\alpha^{2/3} W^+_h$ we find that
	$$\alpha = \alpha^{2/3} \frac{\scal(h)}{6}, $$
	from which we conclude that $\alpha^{1/3}=\frac{\scal(h)}{6}$ and $W^+_g=\mathrm{diag}  ( \alpha , -\alpha/2 , - \alpha/2 )$ and so $|W^+_g|^{2} = \frac{3 \alpha^2}{2}$, which results in
	$$|W^+_g|^{2/3} = \left(\frac{3}{2}\right)^{1/3}  \alpha^{2/3} = \left(\frac{3}{2}\right)^{1/3} \left( \frac{\scal(h)}{6} \right)^2.$$
\end{proof}

\begin{lemma}[Derdzi\'nski]\label{lem:Derdzinsky g Einstein to h}
	Let $h$ be a K\"ahler metric with positive scalar curvature $\scal(h)$, such that the metric $g=(\scal(h))^{-2}h$ is Einstein. Then, $h$ is an extremal K\"ahler metric with vanishing Bach tensor.
\end{lemma}
\begin{proof}
	Let $\sigma:M \to \mathbb{R}$ be a positive function and $g=\sigma^{-2} h$. Then, the Ricci tensors and scalar curvatures of $h$ and $g$ are related by
	\begin{align*}
		\mathrm{Ric}_g & = \mathrm{Ric}_h + 2 \sigma^{-1} \Hess(\sigma) - \left[ \sigma^{-1} \Delta \sigma + 3 \sigma^{-2} |d \sigma|^2 \right] h \\
		\scal(g) & = \sigma^2 \scal(h) - 6 \sigma \Delta \sigma - 12 |d \sigma|^2,
	\end{align*}
	where, on the right hand side, $\Hess$, $\Delta$ and $|\cdot|$ are all taken with respect to $h$.
	
	Now, suppose that $g$ is Einstein, i.e. $\mathrm{Ric}_g=\lambda g.$ Inserting into the above formulas gives
	\begin{align*}
		\lambda \sigma^{-2} h & = \mathrm{Ric}_h + 2 \sigma^{-1} \Hess(\sigma) - \left[ \sigma^{-1} \Delta \sigma + 3 \sigma^{-2} |d \sigma|^2 \right] h ,
	\end{align*}
	which we can rearrange to
	\begin{align*}
		\mathrm{Ric}_h & = - 2 \sigma^{-1} \Hess(\sigma) + \left[ \sigma^{-1} \Delta \sigma + 3 \sigma^{-2} |d \sigma|^2 +  \lambda \sigma^{-2} \right] h .
	\end{align*}
	As for the scalar curvature, we find that
	\begin{align*}
		\scal(h) =  6\sigma^{-1} \Delta \sigma +12\sigma^{-2}|d\sigma|^2 + 4\lambda \sigma^{-2} .
	\end{align*}
	It follows from the equation for the Ricci curvature that if $h$ is K\"ahler, then $\Hess(\sigma)$ is $J$-invariant, as $\mathrm{Ric}_h$ and $h$ both are. In particular, if $\sigma=\scal(h)$ we find that $h$ is actually an extremal K\"ahler metric. In this latter situation we have
	$$(\scal(h))^3 -  6 \scal(h) \Delta \scal(h) - 12|d\scal(h)|^2 =   4\lambda,$$
	and 
	$$\mathrm{Ric}_h = - 2 (\scal(h))^{-1} \Hess(\scal(h)) + \frac{1}{4} \left[ \scal(h) - 2 (\scal(h))^{-1} \Delta \scal(h) \right] h ,$$
	which we can rewrite as
	$$(\mathrm{Ric}_h)_0 = - 2 (\scal(h))^{-1} (\Hess(\scal(h)))_0,$$
	where $(\cdot)_0$ denotes the trace free part of the symmetric tensor $(\cdot)$. The proof is concluded by noticing that Bach tensor of the extremal K\"ahler metric $h$ is given by
	$$B=\frac{1}{12} \left[ \scal(h)(\mathrm{Ric}_h)_0 +2 (\Hess(\scal(h)))_0 \right],$$
	and therefore vanishes for any metric $h$ as above.
\end{proof}

\begin{remark}
	The proof of the previous proposition also shows that if $h$ is an extremal K\"ahler metric, then its scalar curvature $\scal(h)$ satisfies
	\begin{equation}\label{eq:scalar equation for the scalar curvature}
		(\scal(h))^3 -  6 \scal(h) \Delta \scal(h)- 12|d\scal(h)|^2 =   constant.
	\end{equation}
\end{remark}

It turns out that the vanishing of the Bach tensor follows from the condition that $J \nabla \scal(h)$ be a Killing field and the equation
$$(\scal(h))^3 -  6 \scal(h) \Delta \scal(h) - 12|d\scal(h)|^2 =   constant.$$

\begin{lemma}[Derdzi\'nski]\label{lem:Derdzinsky equivalence between vanishing Bach and scalar eq}
	Let $h$ be an extremal K\"ahler metric with scalar curvature $\scal(h)$. Then, the vanishing of the Bach tensor of $h$ is equivalent to $J\nabla \scal(h)$ being a Killing vector field and equation (\ref{eq:scalar equation for the scalar curvature}).
\end{lemma}
\begin{proof}
	For simplicity, in the course of this poof we shall write $\scal(h)=s$ and $\mathrm{Ric}_h = \mathrm{Ric}$. The extremality of $h$ implies that $\nabla^2 s$ is $J$-invariant, i.e. $\nabla^2 s(J \cdot , J \cdot )= \nabla^2 s (\cdot , \cdot)$, which we can write as 
	$$- \nabla^2 s(J \cdot , \cdot ) = \nabla^2 s(J \cdot , J^2 \cdot )= \nabla^2 s (\cdot , J\cdot),$$
	which is equivalent to $\nabla J \nabla s $ being skew-symmetric, i.e. to $J \nabla s$ being Killing.
	
	Next, we compute $\mathrm{Ric}(\nabla s)$ and $\mathrm{Ric}(J\nabla s)$ which are respectively given by
	\begin{align*}
		\mathrm{Ric}_{ik} \nabla_k s & = \nabla_k \nabla_i \nabla_k s - \nabla_i \nabla_k \nabla_k s \\
		& = \nabla_i \Delta s + \nabla_k \nabla_k \nabla_i s ,\\
		\mathrm{Ric}_{ik} J^{kl}\nabla_l s & = \nabla_k \nabla_i J^{kl}\nabla_l s - \nabla_i \nabla_k J^{kl}\nabla_l s \\
		& = - \nabla_k \nabla_k J^{il}\nabla_l s .
	\end{align*}
	where we have used the fact that $J\nabla s$ is Killing and so $\nabla J \nabla s$ is skew-symmetric, i.e. $\nabla_i J^{jl}\nabla_l s + \nabla_j J^{il}\nabla_l s =0$ and in particular $\nabla_k J^{kl}\nabla_l s =0$. Hence, viewing $\mathrm{Ric}$ as an endomorphism of the tangent bundle, we have found that
	\begin{align*}
		\mathrm{Ric} ( \nabla s ) & =  \nabla \Delta s - \mathrm{div} \nabla^2 s ,\\
		\mathrm{Ric} ( J \nabla s) & =  J \mathrm{div} \nabla^2 s .
	\end{align*}
	However, given that $(h,J)$ is a K\"ahler structure $\mathrm{Ric}$ must be $J$-invariant and so $\mathrm{Ric}(\nabla s)= -J \mathrm{Ric}(J\nabla s)$ and so
	$$\nabla \Delta s  = 2 \mathrm{div} \nabla^2 s , $$
	which inserting above yields $\mathrm{Ric}(\nabla s) = \frac{1}{2} \nabla \Delta s$ and plugging into the Bach tensor gives
	\begin{align*}
		12 B(\nabla s) & =  s (\mathrm{Ric})_0 (\nabla s) +2 (\Hess(s))_0 (\nabla s) \\
		& =\left[ s \mathrm{Ric} (\nabla s) + 2 \Hess(s) (\nabla s)  \right] - \left[ \frac{1}{4}s^2 \nabla s - \frac{1}{2} \Delta s \nabla s  \right]  \\
		& = \left[ \frac{1}{2}s \nabla \Delta s +\nabla |\nabla s |^2 \right] - \left[ \frac{1}{4}s^2 \nabla s - \frac{1}{2} \Delta s \nabla s  \right] \\
		& =  \frac{1}{2}s \nabla \Delta s + \frac{1}{2} \Delta s \nabla s +\nabla |\nabla s |^2 - \frac{1}{12} \nabla s^3 \\
		& = \frac{1}{12} \nabla \left( 6 s \Delta s+12 |\nabla s|^2 - s^3 \right).
	\end{align*}
	Thus, under the hypothesis that $ 6 s \Delta s+12 |\nabla s|^2 - s^3$ be constant we find that $B(\nabla s)$ vanishes. Furthermore, as $B$ is trace free and $J$-invariant (for an extremal K\"ahler metric) it must have two double eigenvalues of opposite sign. Hence, we conclude that $B=0$ if $ 6 s \Delta s+12 |\nabla s|^2 - s^3$ is constant. 
	
	The converse follows from the calculation of $B(\nabla s)$.
\end{proof}

\begin{lemma}[Derdzi\'nski]
	Let $h$ be an extremal K\"ahler metric with scalar curvature $\scal(h)$ such that $J\nabla s$ is a Killing vector field and  
	$$(\scal(h))^3 -  6 \scal(h) \Delta \scal(h) - 12|d\scal(h)|^2 =  constant.$$
	Then, $g=\scal(h)^{-2} h$ is Einstein.
\end{lemma}
\begin{proof}
	The proof follows from reversing the argument in the proof of Lemma \ref{lem:Derdzinsky g Einstein to h} using Lemma \ref{lem:Derdzinsky equivalence between vanishing Bach and scalar eq}.
\end{proof}

\begin{remark}
	Derdzi\'nski further states (see remark 4a) in \cite{d}) that the conformal structures containing a K\"ahler metric are precisely those containing metrics which satisfy $\delta W^+=0$. Given a K\"ahler metric that is a unique metric conformal to it which satisfies $\delta W^+=0$. Conversely, given a metric satisfying $\delta W^+=0$ there is a unique metric K\"ahler metric conformal to it. Note that Einstein metrics satisfy $\delta W^+=0$.
\end{remark}

\section{Diagonalising Calabi's metrics}\label{diagonalizing_metric}
The goal of this appendix is to define coordinates that are useful in understanding the extremal $U(2)$-invariant metrics arising from Calabi's ansatz such as those in our Theorem \ref{prop:scalar curvature and extremal metrics}. In particular we seek to diagonalise the metrics induced on the moments polytopes and the metrics induced on the fiber tori. This diagonalised form is likely to be useful in several contexts as well as in future work. Here we use it to study the asymptotic behaviour of the metric we obtain in the zero cone angle limit of the Einstein metrics with cone angle along the divisor at infinity on $\mathbb{CP}^2 \# \overline{ \mathbb{CP}^2}.$ This metrics appears in Proposition \ref{prop: Einstein cone angle}. We will show that it is the smooth Taub-bolt metric.
\begin{lemma}
For any $U(2)$-invariant extremal metric arising from Calabi's ansatz as in Theorem \ref{prop:scalar curvature and extremal metrics},  there are $1$-forms $\eta$, $\kappa$ and $\chi$ such the metric can be written as
$$
\frac{2rdr^2}{\pol}+\frac{\pol}{r}\eta^2+\left(\frac{1}{2rx_1x_2}\kappa^2+2rx_1x_2\chi^2\right)
$$
where $x_1,x_2$ are coordinates on the moment polytope and $r=x_1+x_2$ is invariant under the action of $U(2).$
\end{lemma}
The proof is straightforward.
\begin{proof}
We have $dr=dx_1+dx_2$ so that $\kappa=-x_2dx_1+x_1dx_2$ is orthogonal to $dr.$ In fact
$$
\langle dr,\kappa \rangle=
\begin{pmatrix}
1&1\\
\end{pmatrix}
\begin{pmatrix}
u^{11}&u^{12}\\
u^{12}&u^{22}
\end{pmatrix}
\begin{pmatrix}
-x_2\\
x_1
\end{pmatrix}=-x_2(u^{11}+u^{12})+x_1(u^{22}+u^{12}).
$$
where $u$ is the symplectic potential of the metric, $u_{ij}$ are the entries of $\Hess(u)$ and $u^{ij}$ are the entries of $(\Hess(u))^{-1}$. Now, recall that
	$$
	(\Hess(u))^{-1}=2
	\begin{pmatrix}
		x_1-x_1^2f&-x_1x_2 f\\
		-x_1x_2 f&x_2-x_2^2f\
	\end{pmatrix}.
	$$
where $f$ is a function of $r$ which is given by $f=\frac{q}{r^3}=\frac{r^2-\pol}{r^3}$ with $q$ a quartic polynomial of the form $q(r)=\sum_{i=1}^4 \frac{q_i}{i!} r^i$ as in  Theorem \ref{prop:scalar curvature and extremal metrics} and $\pol(r)=r^2-q(r).$ So 
$$-x_2(u^{11}+u^{12})+x_1(u^{22}+u^{12})=0\implies \langle dr,\kappa \rangle=0.$$
On the other hand 
$$
\begin{pmatrix}
dr\\
\kappa
\end{pmatrix}=
\begin{pmatrix}
1&1\\
-x_2&x_1
\end{pmatrix}
\begin{pmatrix}
dx_1\\
dx_2
\end{pmatrix}
$$
so that 
$$
\begin{pmatrix}
dx_1\\
dx_2
\end{pmatrix}=
\frac{1}{r}\begin{pmatrix}
x_1&-1\\
x_2&1
\end{pmatrix}
\begin{pmatrix}
dr\\
\kappa
\end{pmatrix}.
$$
Replacing $dx_1, dx_2$ in
$$
\sum_{i,j}^2u_{ij}dx_idx_j
$$ we see that
$$
\begin{aligned}
\sum_{i,j}^2u_{ij}dx_idx_j&=&\left(u_{11}+u_{22}-2u_{12}\right)\frac{\kappa^2}{r^2}+\left(x_1^2u_{11}+x_2^2u_{22}+2x_1x_2u_{12}\right)\frac{dr^2}{r^2}\\
&& +\left(x_1u_{11}-x_2^2u_{22}+(x_2-x_1)u_{12}\right)\frac{dr\otimes \kappa}{r^2},\\
&=&\frac{r}{2x_1x_2}\frac{\kappa^2}{r^2}+\frac{2}{r(1-rf)}\frac{dr^2}{r^2},\\
&=&\frac{1}{2rx_1x_2}{\kappa^2}+\frac{2}{\pol}\frac{dr^2}{r},
\end{aligned}
$$
where we have used that
	$$
	(\Hess(u))^{-1}=2
	\begin{pmatrix}
		x_1-x_1^2f&-x_1x_2 f\\
		-x_1x_2 f&x_2-x_2^2f\
	\end{pmatrix}.
	$$
Now set
$$
\begin{cases}
\eta=\frac{x_1d\theta_1+x_2d\theta_2}{r},\\
\chi=\frac{d\theta_1-d\theta_2}{r}.
\end{cases}
$$	
We have 
$$
\begin{pmatrix}
\eta\\
\chi
\end{pmatrix}=
\frac{1}{r}\begin{pmatrix}
x_1&x_2\\
1&-1
\end{pmatrix}
\begin{pmatrix}
d\theta_1\\
d\theta_2
\end{pmatrix}
$$
so that 
$$
\begin{cases}
d\theta_1=x_2\chi+\eta\\
d\theta_2=-x_1\chi+\eta.
\end{cases}
$$
Replacing in 
$$
\sum_{i,j}^2u^{ij}d\theta_id\theta_j,
$$ we see that
$$
\begin{aligned}
\sum_{i,j}^2u_{ij}dx_idx_j&=&\left(x_2^2u^{11}+x_1^2u^{22}-2x_1x_2\right)\chi^2+(u^{11}+2u^{12}+u^{22})\eta^2\\
& &+\left(x_2u^{11}-x_1u^{22}+u^{12}(x_2-x_1)\right)\chi\otimes \eta \\
&=&2r x_1x_2\chi^2+\frac{\pol}{r}\eta^2.&
\end{aligned}
$$
To sum up we see that a $U(2)$-invariant toric metric $g_u$ with symplectic potential $u$ is given by
$$
\frac{2}{r\pol}{dr^2}+\frac{\pol}{r}\eta^2+\frac{1}{2rx_1x_2}{\kappa^2}+2r x_1x_2\chi^2.
$$
\end{proof}

\begin{prop}
The limiting metric from Proposition \ref{prop: Einstein cone angle} i.e. the zero cone angle limit of the cone-angle Einstein metrics we construct on $(\mathbb{CP}^2 \# \overline{ \mathbb{CP}^2} ) \backslash \Sigma$ is toric ALF.  In particular it is the smooth Taub-bolt metric.
\end{prop}
In the proof of Proposition \ref{prop: Einstein cone angle} we showed that the limiting metric has cubic volume growth. The toric ALF condition does not a priori follow from this. But using the fact that the metric is indeed toric ALF it follows from the uniqueness results in \cite{BG} that this metric must be the smooth Taub-bolt metric. In fact the polytope of the limiting metric has 3 smooth edges so it is among the Kerr-Taub-bolt family. On the other hand the moment map image of its edge at infinity is parallel to one of the (smooth) edges of the polytope therefore it must be the smooth Taub-bolt metric (obtained in the $a=0, n\ne 0$ case in the language of \cite{BG}).
\begin{proof}
As we have seen in the proof of Proposition \ref{prop: Einstein cone angle}, the limiting metric is conformal to a K\"ahler toric metric $g$ corresponding to a symplectic potential with associated polynomial
$$
\pol(r)=\frac{a^2}{8}\left(\frac{r^2}{a^2}-1\right)\left(3-\frac{r}{a}\right)^2.
$$
As we saw in the proof of Proposition \ref{prop: Einstein cone angle}  the scalar curvature of such a metric is given by
$$
\scal(g)=\frac{3}{a}\left(3-\frac{r}{a}\right).
$$
The singularity occurs at $r=3a.$ Set $\rho=\left(3-\frac{r}{a}\right)^{-1}$. We have that, close to $r=3a,$
$$
\pol(r)\simeq \frac{a^2}{\rho^{2}}.
$$
We also have $dr=\frac{ad\rho}{\rho^2}$ and $\scal(g)=\frac{3}{a\rho}$. The metric is given by
$$
\frac{g}{\scal^2(g)},
$$
so that according to the above argument, there are $1$-forms $\eta, \chi, \kappa$ such that the metric is given by
$$
\frac{2}{r\pol(r) \scal^2(g)} {dr^2}+\frac{\pol}{r\scal^2(g)}\eta^2+\frac{1}{\scal^2(g)}\left(\frac{1}{2rx_1x_2}{\kappa^2}+2r x_1x_2\chi^2\right).
$$
Now close to $r=3a,$
$$
\frac{\pol}{r\scal^2(g)}\simeq \frac{a^3}{27}
$$
in particular this quantity is bounded. On the other hand 
$$
\frac{2}{r\pol(r) \scal^2(g)}{dr^2}\simeq\frac{2\rho^4}{27a}\frac{a^2d\rho^2}{\rho^4}=\frac{2ad\rho^2}{27}.
$$
The Einstein metric is therefore asymptotic to
$$
\frac{2ad\rho^2}{27}+\frac{a^3 \eta^2}{27}+\frac{a^2 \rho^2}{9}\left(\frac{1}{2rx_1x_2}{\kappa^2}+2r x_1x_2\chi^2\right),
$$
Note that for $R=\frac{\partial}{\partial \theta_1}+\frac{\partial}{\partial \theta_2}$ it can easily be checked that $\eta(R)=1$ and $R\lrcorner\eta=0.$ Also the metric $\left(\frac{1}{2rx_1x_2}{\kappa^2}+2r x_1x_2\chi^2\right)$ has constant curvature (because it is a round metric). We conclude that our Einstein metric toric ALF because $(R,\eta)$ and $\frac{1}{2x_1x_2}{\kappa^2}+2 x_1x_2\chi^2$ are torus invariant.
\end{proof}

\end{document}